\documentclass[12pt]{amsart}
\usepackage{amsmath}
\usepackage{amssymb,latexsym}
\usepackage{mathrsfs}
\usepackage{amsthm}
\usepackage[v2,tips]{xy}
\usepackage[T1]{fontenc}
\usepackage[latin1]{inputenc}

\makeatletter
\@namedef{subjclassname@2010}{%
  \textup{2010} Mathematics Subject Classification}
\makeatother

\newtheorem{theorem}{Theorem}[section]
\newtheorem{lemma}[theorem]{Lemma}
\newtheorem{proposition}[theorem]{Proposition}
\newtheorem{corollary}[theorem]{Corollary}

\theoremstyle{definition}

\newtheorem{example}[theorem]{Example}
\newtheorem{remark}[theorem]{Remark}

\numberwithin{equation}{section}

\newcounter{smallromans}

\newenvironment{romanenumerate}
{\begin{list}{{\normalfont\textrm{(\roman{smallromans})}}}%
  {\usecounter{smallromans}\setlength{\itemindent}{0cm}%
   \setlength{\leftmargin}{5.5ex}\setlength{\labelwidth}{5.5ex}%
   \setlength{\topsep}{.5ex}\setlength{\partopsep}{.5ex}%
   \setlength{\itemsep}{0.1ex}}}%
{\end{list}}

\newcommand{\romanref}[1]{{\normalfont\textrm{(\ref{#1})}}}

\newcounter{smallromansdash}

{\end{list}}

\newcounter{bigromans} 
\newenvironment{capromanenumerate}
{\begin{list}{{\normalfont\textrm{(\Roman{bigromans})}}}%
    {\usecounter{bigromans}\setlength{\itemindent}{0cm}%
      \setlength{\leftmargin}{5.5ex}\setlength{\labelwidth}{6ex}%
      \setlength{\topsep}{.5ex}\setlength{\partopsep}{.5ex}%
      \setlength{\itemsep}{0.1ex}}}%
  {\end{list}}

\newcounter{smallarabics}
{\end{list}}

\newcounter{smallalphs}

\newenvironment{alphenumerate}
{\begin{list}{{\normalfont\textrm{(\alph{smallalphs})}}}%
  {\usecounter{smallalphs}\setlength{\itemindent}{0cm}%
   \setlength{\leftmargin}{5.5ex}\setlength{\labelwidth}{5.5ex}%
   \setlength{\topsep}{.5ex}\setlength{\partopsep}{.5ex}%
   \setlength{\itemsep}{0.1ex}}}%
{\end{list}}

\newcommand{\alphref}[1]{{\normalfont\textrm{(\ref{#1})}}}

\newcommand{\N}{\ensuremath{\mathbb{N}}}

\newcommand{\C}{\ensuremath{\mathbb{C}}}

\renewcommand{\phi}{\ensuremath{\varphi}}
\renewcommand{\epsilon}{\ensuremath{\varepsilon}}

\renewcommand{\leq}{\ensuremath{\leqslant}}
\renewcommand{\le}{\ensuremath{\leqslant}}
\renewcommand{\geq}{\ensuremath{\geqslant}}
\renewcommand{\ge}{\ensuremath{\geqslant}}

\newcommand{\smashw}[2][l]{{\text{\makebox[0pt][#1]{$#2$}}}}

\begin{document}
\title[Maximal left ideals of operators on a Banach space]{Maximal
  left ideals of the Banach algebra of bounded operators on a Banach
  space}
\author{H.~G.~Dales}
\address{Department of Mathematics and
  Statistics, Fylde College\\
  Lancaster University,
  Lancaster LA1 4YF, UK}
\email{g.dales@lancaster.ac.uk}
\author[T.~Kania]{Tomasz Kania}
\address{Department of Mathematics and
  Statistics, Fylde College\\
  Lancaster University,
  Lancaster LA1 4YF, UK}
\email{t.kania@lancaster.ac.uk}
\author[T.~Kochanek]{Tomasz Kochanek}
\address{Institute of Mathematics, University of Silesia\\
  ul.\ Bankowa 14, 40-007 Katowice, Poland} 
\email{t\_kochanek@wp.pl}
\author[P.~Koszmider]{Piotr Koszmider}
\address{Institute of Mathematics, Polish Academy of Sciences\\ 
ul.\ \'{S}niadeckich~8, 00-956 Warszawa, Poland}
\email{P.Koszmider@Impan.pl}
\author[N.~J.~Laustsen]{Niels Jakob Laustsen}
\address{Department of Mathematics and
  Statistics, Fylde College\\
  Lancaster University,
  Lancaster LA1 4YF, UK} 
\email{n.laustsen@lancaster.ac.uk}
\date{}

\begin{abstract} We address the following two
  questions regarding the maximal left ideals of the Banach
  algebra~$\mathscr{B}(E)$ of bounded operators acting on an
  infinite-dimensional Banach space~$E\colon$
  \begin{capromanenumerate}
  \item\label{abstractQ1} Does~$\mathscr{B}(E)$ always contain a
    maximal left ideal which is not finitely generated?
  \item\label{abstractQ2} Is every finitely-generated, maximal left
    ideal of~$\mathscr{B}(E)$ necessarily of the form
    \begin{equation}\label{abstractEq1} \{ T\in\mathscr{B}(E) : Tx =
      0\} \tag{$*$} \end{equation} for some non-zero $x\in E$?
  \end{capromanenumerate}
  Since the two-sided ideal~$\mathscr{F}(E)$ of finite-rank operators
  is not contained in any of the maximal left ideals given
  by~\eqref{abstractEq1}, a positive answer to the second question
  would imply a~positive answer to the first.

  Our main results are: (i)~Question~\romanref{abstractQ1} has a
  positive answer for most (possibly all) in\-finite-di\-men\-sional
  Banach spaces; (ii)~Question~\romanref{abstractQ2} has a positive
  answer if and only if no finitely-generated, maximal left ideal
  of~$\mathscr{B}(E)$ contains~$\mathscr{F}(E)$; (iii)~the answer to
  Question~\romanref{abstractQ2} is positive for many, but not all,
  Banach spaces.\bigskip

\noindent
  To appear in \emph{Studia Mathematica.}
\end{abstract}
\maketitle
\section{Introduction and statement of main results}
\noindent
The purpose of this paper is to study the maximal left ideals of the
Banach algebra~$\mathscr{B}(E)$ of (bounded, linear) operators acting
on a Banach space~$E$, particularly the maximal left ideals that are
finitely generated. A general introduction to the Banach
algebra~$\mathscr{B}(E)$ can be found in~\cite[\S2.5]{dales}. Our
starting point is the elementary observation that $\mathscr{B}(E)$
always contains a large supply of singly-generated, maximal left
ideals, namely
\begin{equation}\label{trivialfgmaxrightidealsEq}
  \mathscr{M}\!\mathscr{L}_x = \{ T\in\mathscr{B}(E) : Tx =
  0\}\quad (x\in E\setminus\{0\}) \end{equation} (see
Proposition~\ref{trivialfgmaxrightideals} for details).  We call the
maximal left ideals of this form \emph{fixed}, inspired by the
analogous terminology for ultrafilters.

The Banach algebra $\mathscr{B}(E)$ is semi\-simple, as is well known
(\emph{e.g.}, see~\cite[Theorem~2.5.8]{dales}); that is, the
intersection of its maximal left ideals is $\{0\}$. We observe that
this is already true for the inter\-section of the fixed maximal left
ideals.

In the case where the Banach space $E$ is finite-dimensional, an
elementary result in linear algebra states that the mapping
\[ F\mapsto \{T\in\mathscr{B}(E) : F\subseteq\ker T\} \] is an
anti-isomorphism of the lattice of linear subspaces of~$E$ onto the
lattice of left ideals of~$\mathscr{B}(E)$ (\emph{e.g.},
see~\cite[p.~173, Exercise~3]{jacobson}). Hence each maximal left
ideal~$\mathscr{L}$ of~$\mathscr{B}(E)$ corresponds to a unique
minimal, non-zero linear subspace of~$E$, that is, a one-dimensional
subspace, and therefore~$\mathscr{L}$ is fixed. This conclusion is
also an easy consequence of our work, as outlined in
Remark~\ref{remarkFinitedimCase}\romanref{remarkFinitedimCaseB},
below. In contrast, this statement is false whenever~$E$ is
infinite-dimensional because the two-sided ideal~$\mathscr{F}(E)$ of
finite-rank operators is proper, so that, by Krull's theorem, it is
contained in a maximal left ideal, which cannot be fixed since, for
each $x\in E\setminus\{0\}$, there is a finite-rank operator~$T$
on~$E$ such that $Tx\neq 0$.

Inspired by these observations and his
collaboration~\cite{daleszelazko} with \.{Z}elazko, the first-named
author raised the following two questions for an infinite-dimensional
Banach space~$E\colon$
\begin{capromanenumerate}
\item\label{DalesQ2} Does~$\mathscr{B}(E)$ always contain a maximal
  left ideal which is not finitely generated?
\item\label{DalesQ1} Is every finitely-generated, maximal left ideal
  of~$\mathscr{B}(E)$ necessarily fixed?
\end{capromanenumerate}
In the light of the previous paragraph, we note that a positive answer
to~\romanref{DalesQ1} would imply a positive answer
to~\romanref{DalesQ2}.

The answers to the above questions depend only on the isomorphism
class of the Banach space~$E$. This follows from the theorem of
Eidelheit, which states that two Banach spaces~$E$ and~$F$ are
isomorphic if and only if the corresponding Banach
algebras~$\mathscr{B}(E)$ and~$\mathscr{B}(F)$ are isomorphic
(\emph{e.g.}, see \cite[Theorem~2.5.7]{dales}).  \smallskip

After presenting some preliminary material in
Section~\ref{SectPrelim}, we shall use a counting argu\-ment in
Section~\ref{section_countingmaxideals} to answer
Question~\romanref{DalesQ2} positively for a large class of Banach
spaces, in\-clud\-ing all separable Banach spaces which contain an
infinite-dimensional, closed, complemented subspace with an
unconditional basis and, more generally, all separable Banach spaces
with an unconditional Schauder decomposition (see
Corollary~\ref{countablescauderdecomp} for details).

We then turn our attention to Question~\romanref{DalesQ1}.
Section~\ref{sectionLcontainsFE} begins with the following dichotomy,
which can be viewed as the analogue of the fact that an ultrafilter on
a set~$M$ is either fixed (in the sense that it consists of precisely
those subsets of~$M$ which contain a fixed element $x\in M$), or it
contains the Fr\'{e}chet filter of all cofinite subsets of~$M$.
\begin{theorem}[Dichotomy for maximal left ideals]\label{dichotomythm}
  Let~$E$ be a non-zero Banach space. Then, for each maximal left
  ideal~$\mathscr{L}$ of~$\mathscr{B}(E)$, exactly one of the
  following two alternatives holds:
  \begin{romanenumerate}
  \item\label{dichotomythm1} $\mathscr{L}$ is fixed; or
  \item\label{dichotomythm2} $\mathscr{L}$ contains $\mathscr{F}(E)$.
  \end{romanenumerate}
\end{theorem}

\begin{remark}\label{remarkFinitedimCase}
\begin{romanenumerate}
\item\label{remarkFinitedimCaseB} Since \mbox{$\mathscr{F}(E) =
    \mathscr{B}(E)$} for each finite-dimensional Banach space~$E$, no
  proper left ideal of~$\mathscr{B}(E)$ satisfies
  condition~\romanref{dichotomythm2} of Theorem~\ref{dichotomythm}.
  Hence, by this theorem, each maximal left ideal of~$\mathscr{B}(E)$
  is fixed whenever~$E$ is finite-dimensional.
\item Another immediate consequence of Theorem~\ref{dichotomythm} is
  that Question~\romanref{DalesQ1} has a positive answer for a Banach
  space~$E$ if and only if~$\mathscr{F}(E)$ is not contained in any
  finitely-generated, maximal left ideal of~$\mathscr{B}(E)$.
\item In Corollary~\ref{June12dich1}, below, we shall deduce from
  Theorem~\ref{dichotomythm} a slightly stronger, but also more
  technical, conclusion that involves the larger ideal of inessential
  operators instead of~$\mathscr{F}(E)$.
\end{romanenumerate}
\end{remark}

The other main result to be proved in Section~\ref{sectionLcontainsFE}
is the following dichotomy for closed left ideals of~$\mathscr{B}(E)$
that are not necessarily maximal.

\begin{theorem}[Dichotomy for closed left ideals]\label{dich2}
  Let~$E$ be a non-zero Banach space, let~$\mathscr{L}$ be a closed
  left ideal of~$\mathscr{B}(E)$, and suppose that~$E$ is reflexive
  or~$\mathscr{L}$ is finitely generated. Then exactly one of the
  following two alternatives holds:
  \begin{romanenumerate}
  \item\label{dich2a} $\mathscr{L}$ is contained in a fixed maximal
    left ideal; or
  \item\label{dich2b} $\mathscr{L}$ contains~$\mathscr{F}(E)$.
  \end{romanenumerate}
\end{theorem}

We note that Theorems~\ref{dichotomythm} and~\ref{dich2} are genuine
dichotomies, in the sense that in both theorems the two
alternatives~\romanref{dich2a} and~\romanref{dich2b} are mutually
exclusive because, as observed above, no fixed maximal left ideal
of~$\mathscr{B}(E)$ contains~$\mathscr{F}(E)$.

The purpose of Sections~\ref{sectionAllfgidealsfixed}
and~\ref{newSection6} is to show that Question~\romanref{DalesQ1} has
a positive answer for many Banach spaces, both `classical' and more
`exotic' ones. We can summarize our results as follows, and refer to
Sections~\ref{sectionAllfgidealsfixed} and~\ref{newSection6} for full
details, including precise definitions of any unexplained terminology.

\begin{theorem}\label{allmaxleftidealsarefixed} 
  Let $E$ be a Banach space which satisfies one of the following six
  condi\-tions:
  \begin{romanenumerate}
  \item $E$ has a Schauder basis and is complemented in its bidual;
  \item $E$ is isomorphic to the dual space of a Banach space with a
    Schauder basis;
  \item $E$ is an injective Banach space;
  \item\label{allmaxleftidealsarefixed4} $E = c_0(\mathbb{I})$, $E =
    H$, or $E = c_0(\mathbb{I})\oplus H$, where~$\mathbb{I}$ is a
    non-empty index set and $H$ is a Hilbert space;
  \item $E$ is a Banach space which has few operators;
  \item $E = C(K)$, where $K$ is a compact Hausdorff space without
    isolated points, and each operator on $C(K)$ is a weak
    multiplication.
  \end{romanenumerate}
  Then each finitely-generated, maximal left ideal of~$\mathscr{B}(E)$
  is fixed.
\end{theorem}

On the other hand, there is a Banach space for which the answer to
Question~\romanref{DalesQ1} is nega\-tive; this is the main result of
Section~\ref{nonfixedmaxleftideal}. Its statement involves
Argyros--Hay\-don's Banach space having very few operators. We denote
this space by~$X_{\text{\normalfont{AH}}}$, and refer to
Theorem~\ref{thmAH} for a summary of its main properties.

\begin{theorem}\label{Thmnonfixedmaxleftideal} Let
  $E=X_{\text{\normalfont{AH}}}\oplus\ell_\infty$. Then the set
  \begin{equation}\label{idealK1}
    \mathscr{K}_1 = \left\{\begin{pmatrix} T_{1,1} & T_{1,2}\\
        T_{2,1} & 
        T_{2,2} \end{pmatrix}\in\mathscr{B}(E) :
      T_{1,1}\ {\normalfont{\text{is compact}}}\right\} 
  \end{equation}  
  is a maximal two-sided ideal of codimension one in~$\mathscr{B}(E)$,
  and hence also a maximal left ideal. Moreover, $\mathscr{K}_1$ is
  singly generated as a left ideal, and it is not fixed.
\end{theorem}
This theorem suggests   that the Banach space
$E=X_{\text{\normalfont{AH}}}\oplus\ell_\infty$ is a natural candidate
for providing a negative answer to
Question~\romanref{DalesQ2}. However, as we shall also show in
Section~\ref{nonfixedmaxleftideal}, it does not.
\begin{theorem}\label{Thmnonfixedmaxleftideal2} 
  Let $E=X_{\text{\normalfont{AH}}}\oplus\ell_\infty$. Then the ideal
  $\mathscr{K}_1$ given by~\eqref{idealK1} is the unique non-fixed,
  finitely-generated, maximal left ideal of~$\mathscr{B}(E)$.  Hence
  \begin{equation}\label{idealW2}
    \mathscr{W}_2 = \left\{\begin{pmatrix} T_{1,1} & T_{1,2}\\
        T_{2,1} & 
        T_{2,2} \end{pmatrix}\in\mathscr{B}(E) :
      T_{2,2}\ {\normalfont{\text{is weakly compact}}}\right\},
  \end{equation}  
  which is a maximal two-sided ideal of~$\mathscr{B}(E)$, is not
  contained in any finitely-generated, maximal left ideal
  of~$\mathscr{B}(E)$.
\end{theorem}

To conclude this summary of our results, let us point out that
Question~\romanref{DalesQ1} remains open in some important cases,
notably for $E = C(K)$, where $K$ is any infinite, compact metric
space such that $C(K)$ is not isomorphic to~$c_0$. \smallskip

As a final point, we shall explain how our work fits into a more
general context.  The main motivation behind
Question~\romanref{DalesQ2} comes from the fact that it is the special
case where $\mathscr{A} = \mathscr{B}(E)$ for a Banach space~$E$ of
the following conjecture, raised and discussed in~\cite{daleszelazko}:
\begin{quote}\label{DZconj}
  \textit{Let~$\mathscr{A}$ be a unital Banach algebra such that every
    maximal left ideal of~$\mathscr{A}$ is finitely generated. Then
    $\mathscr{A}$ is finite-dimensional.}
\end{quote}
A stronger form of this conjecture in the case where~$\mathscr{A}$ is
commutative was proved by Ferreira and Tomassini~\cite{ft}; extensions
of this result are given in~\cite{daleszelazko}. The conjecture is
also known to be true for $C^*$\nobreakdash-alge\-bras. For a proof
and a generalization of this result to the class of Hilbert
$C^*$\nobreakdash-modules, see~\cite{BlecherKania}.

The above conjecture was suggested by Sinclair and Tullo's
theorem~\cite{st}, which states that a Banach algebra~$\mathscr{A}$ is
finite-dimensional if each closed left ideal of~$\mathscr{A}$ (not
just each maximal one) is finitely generated.  This result has been
generalized by Boudi~\cite{boudi}, who showed that the conclusion
that~$\mathscr{A}$ is finite-dimensional remains true under the
formally weaker hypo\-thesis that each closed left ideal
of~$\mathscr{A}$ is countably generated. (Boudi's theorem can in fact
be deduced from Sinclair and Tullo's theorem because a closed,
countably-generated left ideal is necessarily finitely generated
by~\cite[Prop\-o\-si\-tion~1.5]{daleszelazko}.)

Another result that is related to our general theme, but of a
different flavour from those just mentioned, is due to Gr\o{}nb\ae{}k
\cite[Proposition~7.3]{GronbaekMorita}, who has shown that, for a
Banach space~$E$ with the approximation property, the mapping
\[ F\mapsto \overline{\operatorname{span}}\{x\otimes\lambda : x\in
E,\, \lambda\in F\} \] is an isomorphism of the lattice of closed
linear subspaces~$F$ of the dual space of~$E$ onto the lattice of
closed left ideals of the Banach algebra of compact operators on~$E$.

\section{Preliminaries}\label{SectPrelim}
\noindent
Our notation is mostly standard.  We write $|M|$ for the cardinality
of a set~$M$. As usual, $\aleph_0$ and~$\aleph_1$ denote the first and
second infinite cardinals, respectively, while $\mathfrak{c} =
2^{\aleph_0}$ is the cardinality of the continuum.

Let $E$ be a Banach space, always supposed to be over the complex
field~$\mathbb{C}$. We denote by~$I_E$ the identity operator on~$E$.
For a non-empty set~$\mathbb{I}$, we define
\[ \ell_\infty(\mathbb{I},E) = \{ f\colon\mathbb{I}\to E : \|
f\|_\infty<\infty\},\quad\text{where}\quad \|f\|_\infty =
\sup_{i\in\mathbb{I}}\|f(i)\|, \] so that $\ell_\infty(\mathbb{I},E)$
is a Banach space with respect to the norm~$\|\cdot\|_\infty$.  The
following special conventions apply:
\begin{itemize}
\item $\ell_\infty(\mathbb{I}) = \ell_\infty(\mathbb{I},\C)$;
\item $\ell_\infty = \ell_\infty(\N);$
\item $E^n = \ell_\infty\bigl(\{1,\ldots,n\},E\bigr)$ for each
  $n\in\N$.
\end{itemize}

We write $E^*$ for the (continuous) dual space of the Banach
space~$E$. The duality bracket between~$E$ and~$E^*$ is
$\langle\,\cdot\,,\,\cdot\,\rangle$, while $\kappa_E\colon E\to
E^{**}$ denotes the canonical embedding of~$E$ into its bidual. By an
\emph{operator} we understand a bounded, linear operator between
Banach spaces; we write $\mathscr{B}(E,F)$ for the Banach space of all
operators from~$E$ to another Banach space~$F$, and denote by $T^*\in
\mathscr{B}(F^*,E^*)$ the adjoint of an
operator~$T\in\mathscr{B}(E,F)$.

We shall require the following standard notions for
\mbox{$T\in\mathscr{B}(E,F)\colon$}
\begin{romanenumerate}
\item\label{OpProp1} $T$ is a \emph{finite-rank operator} if it has
  finite-dimensional range. We write $\mathscr{F}(E,F)$ for the set of
  finite-rank operators from~$E$ to~$F$.  It is well known that
  \begin{equation*}
    \mathscr{F}(E,F) = \operatorname{span}\{ y\otimes \lambda : y\in
    F,\,\lambda\in E^*\},
  \end{equation*}
  where $y\otimes \lambda$ denotes the rank-one operator given by
   \[ y\otimes \lambda\colon\ 
  x\mapsto\langle x,\lambda\rangle y,\quad E\to F\quad (y\in F,\,
  \lambda\in E^*). \] The following elementary observation
   will be used several times:
  \begin{equation}\label{eqCompRankOneOp}
    \mbox{}\quad    R(y\otimes\lambda)S = (Ry)\otimes (S^*\lambda)\quad
    (\,y\in F,\,\lambda\in E^*),
  \end{equation}
  valid for any Banach spaces $D$, $E$, $F$, and $G$ and any
  $S\in\mathscr{B}(D,E)$ and $R\in\mathscr{B}(F,G)$.
\item\label{OpProp2} $T$ is \emph{compact} if the image under~$T$ of
  the unit ball of~$E$ is a relatively norm-compact subset of~$F$. We
  write $\mathscr{K}(E,F)$ for the set of compact operators from~$E$
  to~$F$.
\item\label{OpProp3} $T$ is \emph{weakly compact} if the image
  under~$T$ of the unit ball of~$E$ is a relatively weakly compact
  subset of~$F$. We write $\mathscr{W}(E,F)$ for the set of weakly
  compact operators from~$E$ to~$F$.
\item\label{OpProp4} $T$ is \emph{bounded below} if, for some
  $\epsilon>0$, we have $\|Tx\|\ge\epsilon\|x\|$ for each $x\in E$;
  or, equivalently, $T$ is injective and has closed range.  This
  notion is dual to surjectivity in the following precise sense
  (\emph{e.g.}, see \cite[Theorem~3.1.22]{meg}):
  \begin{equation}\label{dualitySurBB}
    \begin{array}{rcl}
      T\ \text{is surjective}\ \ &\Longleftrightarrow&\ \ T^*\
      \text{is bounded below,}\\
      T\ \text{is bounded below}\ \ &\Longleftrightarrow&\ \ T^*\
      \text{is 
        surjective.} \end{array} 
  \end{equation}
\item\label{OpProp5} $T$ is \emph{strictly singular} if no restriction
  of~$T$ to an infinite-dimensional subspace of~$E$ is bounded below;
  that is, for each $\epsilon >0$, each infinite-dimensional subspace
  of~$E$ contains a unit vector~$x$ with  $\|Tx\|\le\epsilon$.  We
  write $\mathscr{S}(E,F)$ for the set of strictly singular operators
  from~$E$ to~$F$.
\item\label{OpProp6} $T$ is a \emph{Fredholm operator} if its kernel
  is finite-dimensional and its range is finite-codimensional, in
  which case $T$ has closed range. 
\item\label{OpProp8} $T$ is an \emph{upper semi-Fredholm operator} if
  it has finite-dimensional kernel and closed range. 
\item\label{OpProp7} $T$ is \emph{inessential} if $I_E - ST$ is a
  Fredholm operator for each \mbox{$S\in\mathscr{B}(F,E)$}. We write
  $\mathscr{E}(E,F)$ for the set of inessential operators from~$E$
  to~$F$.
\end{romanenumerate}
The six classes~$\mathscr{B}$, $\mathscr{F}$, $\mathscr{K}$,
$\mathscr{W}$, $\mathscr{S}$, and~$\mathscr{E}$ introduced above
define operator ideals in the sense of Pietsch~\cite{pi}, all of which
except $\mathscr{F}$ are closed.  The following inclusions always
hold:
\begin{align*}
  \mathscr{F}(E,F)\subseteq \mathscr{K}(E,F) &\subseteq
  \mathscr{W}(E,F)\cap\mathscr{S}(E,F)\\ &\subseteq
  \mathscr{S}(E,F)\subseteq \mathscr{E}(E,F)\subseteq
  \mathscr{B}(E,F);
\end{align*}
no others are true in general (even in the case where $E=F$).  In line
with common practice, we set $\mathscr{I}(E) = \mathscr{I}(E,E)$ for
each of the above operator ideals~$\mathscr{I}$.

We remark that a left, right, or two-sided ideal of~$\mathscr{B}(E)$
is proper if and only if it does not contain the identity
operator~$I_E$.  This shows in particular that the two-sided ideal
$\mathscr{E}(E)$ (and hence also $\mathscr{F}(E)$, $\mathscr{K}(E)$,
and $\mathscr{S}(E)$) is proper whenever $E$ is infinite-dimensional.

\begin{remark}\label{KleineckeDefnIness}
The ideal of inessential operators on a single Banach space~$E$ was
originally introduced by Kleinecke~\cite{kl} as the preimage of the
Jacobson radical of the Calkin algebra
$\mathscr{B}(E)\big/\,\overline{\mathscr{F}(E)}$, following
Yood's observation~\cite[p.~615]{yo} that this radical may be
non-zero. The early theory of inessential operators is expounded in
the monograph~\cite{cpy} of Caradus, Pfaffen\-berger, and
Yood. Pietsch~\cite{pi} subsequently gave the `operator ideal'
definition of $\mathscr{E}(E,F)$ stated in~\romanref{OpProp7}, above,
and showed that it coincides with Kleinecke's original definition in
the case where $E=F$.
\end{remark}

The following notion is central to this paper.  Let~$\Gamma$ be a
non-empty subset of~$\mathscr{B}(E)$ for some Banach space~$E$.  The
\emph{left ideal generated by}~$\Gamma$ is the smallest left
ideal~$\mathscr{L}_{\Gamma}$ of~$\mathscr{B}(E)$ that
contains~$\Gamma$. It can be described explicitly as
\begin{equation}\label{eqLeftidealgenbyfam}
  \mathscr{L}_{\Gamma} = \biggl\{ \sum_{j=1}^n S_jT_j :  
  S_1,\ldots,S_n\in\mathscr{B}(E),\,
  T_1,\ldots,T_n\in\Gamma, \,n\in\N\biggr\}.
\end{equation}
A left ideal $\mathscr{L}$ of~$\mathscr{B}(E)$ is \emph{singly}
(respectively, \emph{finitely, countably}) \emph{gen\-er\-ated} if
$\mathscr{L} = \mathscr{L}_{\Gamma}$ for some singleton (respectively,
non-empty and finite, countable) subset~$\Gamma$ of~$\mathscr{B}(E)$.

In the case where~$\Gamma$ is a non-empty, norm-bounded subset
of~$\mathscr{B}(E)$, we can define an operator~$\Psi_{\Gamma}\colon
E\to\ell_\infty(\Gamma,E)$ by
\begin{equation}\label{defnPsiInfinite}
  (\Psi_\Gamma x)(T) = Tx\quad (x\in E,\, T\in\Gamma). \end{equation}
In particular, when~$\Gamma$ is finite, say $\Gamma =
\{T_1,\ldots,T_n\}$, where $n\in\N$ and $T_1,\ldots,T_n$ are distinct,
we shall identify~$\ell_\infty(\Gamma,E)$ with~$E^n$ in the natural way. Then
\begin{equation}\label{defnOpPsi} \Psi_\Gamma =
  \sum_{j=1}^n \iota_j T_j\in\mathscr{B}(E,E^n),
\end{equation}
where $\iota_j\colon E\to E^n$ denotes the canonical $j^{\text{th}}$
coordinate embedding, and \eqref{eqLeftidealgenbyfam} simplifies to
\begin{equation}\label{eqFGLeftideal04072012}
  \mathscr{L}_{\Gamma} = \biggl\{ \sum_{j=1}^n S_jT_j :
  S_1,\ldots,S_n\in\mathscr{B}(E)\biggr\} = \bigl\{ S\Psi_{\Gamma} :
  S\in\mathscr{B}(E^n,E)\bigr\}.
\end{equation} 
The operator~$\Psi_\Gamma$ will play a key role in our work.  We
shall give here only one, very simple, application of~$\Psi_\Gamma$,
showing that each finitely-generated left ideal of operators is
already singly generated for most `classical' Banach spaces.

\begin{proposition}\label{cartesianSGnew}
  Let $E$ be a Banach space which contains a complemented subspace
  that is isomorphic to~$E\oplus E$. Then each finitely-generated left
  ideal of~$\mathscr{B}(E)$ is singly generated.
\end{proposition}
\begin{proof} Let $\Gamma$ be a non-empty, finite subset
  of~$\mathscr{B}(E)$ with $n=|\Gamma|\in\N$.  By the assumption, $E$
  contains a complemented subspace which is isomorphic to~$E^n$, and
  hence there are operators $U\in\mathscr{B}(E^n,E)$ and
  $V\in\mathscr{B}(E,E^n)$ such that $I_{E^n} = VU$. We shall now
  complete the proof by showing that the left ideal
  $\mathscr{L}_{\Gamma}$ is generated by the single operator $T =
  U\Psi_{\Gamma}\in\mathscr{B}(E)$.

  By~\eqref{eqFGLeftideal04072012}, we have
  $T\in\mathscr{L}_{\Gamma}$, so
  that~$\mathscr{L}_{\{T\}}\subseteq\mathscr{L}_{\Gamma}$.

  Conversely, each  $R\in\mathscr{L}_{\Gamma}$ has the form $R
  = S\Psi_{\Gamma}$ for some $S\in\mathscr{B}(E^n,E)$
  by~\eqref{eqFGLeftideal04072012}, and therefore $R =
  S(VU)\Psi_{\Gamma} = (SV)T\in\mathscr{L}_{\{T\}}$.
\end{proof}

\begin{remark} Not all finitely-generated, maximal left ideals in a
  Banach algebra are singly generated.  For instance, let $$\mathbb{B}
  = \{(z,w)\in \C^2 :| z|^2 + |w|^2\leq 1\}$$ be the closed unit ball
  in $\C^2$, and consider the `polyball algebra'~$\mathscr{A}$
  on~$\mathbb{B}$, so that by definition~$\mathscr{A}$ is the closure
  with respect to the uniform norm of the polynomials in two variables
  restricted to~$\mathbb{B}$.  It is shown in
  \cite[Ex\-am\-ple~15.8]{stout} that the maximal ideal
  $\mathscr{M}=\{f\in \mathscr{A}: f(0,0) =0\}$ of~$\mathscr{A}$ is
  generated by the two coordinate functionals, but on the other hand,
  it is clear that~$\mathscr{M}$ is not singly generated.

  It is significantly harder to find a Banach space~$E$ such
  that~$\mathscr{B}(E)$ contains a maximal left ideal which is
  finitely, but not singly, generated.  Such an example has, however,
  recently been obtained~\cite{kanialaustsen}.
\end{remark}

Our next result collects some basic facts about the fixed maximal left
ideals of~$\mathscr{B}(E)$, most of which were already stated in the
Introduction.

\begin{proposition}\label{trivialfgmaxrightideals}
  Let $x$ and $y$ be non-zero elements of a Banach space~$E$. Then:
  \begin{romanenumerate}
  \item\label{trivialfgmaxrightideals1} the set
    $\mathscr{M}\!\mathscr{L}_x$ given by
    \eqref{trivialfgmaxrightidealsEq} is the left ideal of
    $\mathscr{B}(E)$ generated by the projection
    \mbox{$I_E-x\otimes\lambda$}, where $\lambda\in E^*$ is any
    functional such that $\langle x,\lambda\rangle = 1;$
  \item\label{trivialfgmaxrightideals2} the left ideal
    $\mathscr{M}\!\mathscr{L}_x$ is maximal;
  \item\label{trivialfgmaxrightideals3} $\mathscr{M}\!\mathscr{L}_x =
    \mathscr{M}\!\mathscr{L}_y$ if and only if $x$ and $y$ are
    proportional.
  \end{romanenumerate}
  In particular, $\mathscr{B}(E)$ contains~$|E|$ distinct, fixed
  maximal left ideals whenever $E$ is infinite-dimensional.
\end{proposition}
\begin{proof} \romanref{trivialfgmaxrightideals1}.  Let $P =
  I_E-x\otimes\lambda$. The set $\mathscr{M}\!\mathscr{L}_x$ is
  clearly a left ideal which contains~$P$, and hence
  $\mathscr{L}_{\{P\}}\subseteq \mathscr{M}\!\mathscr{L}_x$. The
  reverse inclusion holds because
  $T(x\otimes\lambda) = 0$ for each $T\in\mathscr{M}\!\mathscr{L}_x$,
  so that $T = TP\in \mathscr{L}_{\{P\}}$.

  \romanref{trivialfgmaxrightideals2}. The left ideal
  $\mathscr{M}\!\mathscr{L}_x$ is evidently proper.  To verify that it
  is maximal, suppose that $T\in \mathscr{B}(E)\setminus
  \mathscr{M}\!\mathscr{L}_x$. Then $Tx\neq 0$, so that $\langle
  Tx,\mu\rangle = 1$ for some $\mu\in E^*$. The operator $S
  = I_E - (x\otimes\mu)T$ belongs to~$\mathscr{M}\!\mathscr{L}_x$
  because $(x\otimes\mu)Tx = \langle Tx,\mu\rangle x = x$, and
  consequently \[ I_E = S +
  (x\otimes\mu)T\in\mathscr{M}\!\mathscr{L}_x +
  \mathscr{L}_{\{T\}}. \]

  \romanref{trivialfgmaxrightideals3}. It is clear that
  $\mathscr{M}\!\mathscr{L}_x = \mathscr{M}\!\mathscr{L}_y$ if $x$ and
  $y$ are proportional. We prove the converse by
  contraposition. Suppose that $x$ and $y$ are linearly
  independent. Then we can take $\lambda\in E^*$ such that
  $\langle x,\lambda\rangle = 1$ and $\langle y,\lambda\rangle = 0$,
  and hence $x\otimes\lambda\in\mathscr{M}\!\mathscr{L}_y\setminus
  \mathscr{M}\!\mathscr{L}_x$.
\end{proof}

We conclude this preliminary section with the observation that the
answer to the analogue of Question~\romanref{DalesQ2} for two-sided
ideals is negative, as the following example shows.

\begin{example} Consider the Hilbert space $H = \ell_2(\aleph_1)$, 
  and take a projection $P\in\mathscr{B}(H)$ with separable,
  infinite-dimensional range.  The ideal classification of
  Gramsch~\cite{gr} and Luft~\cite{luft} implies that the
  ideal~$\mathscr{X}(H)$ of operators with separable range is the
  unique maximal two-sided ideal of~$\mathscr{B}(H)$. Given
  $T\in\mathscr{X}(H)$, let $Q\in\mathscr{B}(H)$ be the orthogonal
  projection onto~$\overline{T(H)}$. Then $T = QT$, and also $Q = VPU$
  for some operators $U,V\in\mathscr{B}(H)$, so that $T =
  VPUT$. Hence~$\mathscr{X}(H)$ is the two-sided ideal
  of~$\mathscr{B}(H)$ generated by the single operator~$P$.
  Since~$\mathscr{X}(H)$ is the only maximal two-sided ideal
  of~$\mathscr{B}(H)$, we conclude that {each} maximal two-sided ideal
  of~$\mathscr{B}(H)$ is singly generated, and therefore the analogue
  of Question~\romanref{DalesQ2} for two-sided ideals has a negative
  answer.

  With slightly more work, we can give a similar example based on a
  separable Banach space. To this end, consider the $p^{\text{th}}$
  quasi-reflexive James space~$J_p$ for some
  $p\in(1,\infty)$. Edelstein and Mityagin~\cite{em} observed that the
  two-sided ideal~$\mathscr{W}(J_p)$ of weakly compact opera\-tors is
  maximal because it has codimension one in~$\mathscr{B}(J_p)$,
  and~$\mathscr{B}(J_p)$ contains no other maximal two-sided ideals by
  \cite[Theorem~4.16]{lau1}.  We shall now show
  that~$\mathscr{W}(J_p)$ is singly generated as a two-sided
  ideal. Let \[ J_p^{(n)} = \bigl\{(\alpha_j)_{j\in\N}\in J_p :
  \alpha_j = 0\ (j>n)\bigr\}\quad (n\in\N).\] Then the Banach space
  \[ J_p^{(\infty)} = \biggl(\bigoplus_{n\in\N}
  J_p^{(n)}\biggr)_{\ell_p} \] is reflexive and isomorphic to a
  complemented subspace of~$J_p$. (The latter observation is due to
  Edelstein and Mityagin~\cite[Lemma~6(d)]{em}; an alter\-na\-tive
  approach can be found in~\cite[Proposition~4.4(iv)]{lau1}.)  Take a
  projec\-tion \mbox{$P\in\mathscr{B}(J_p)$} whose range is isomorphic
  to $J_p^{(\infty)}$. By~\cite[Theorem~4.3]{lau2}, we have
  \begin{align*} \mathscr{W}(J_p) &= \bigl\{ TS :
    S\in\mathscr{B}(J_p,J_p^{(\infty)}),\,
    T\in\mathscr{B}(J_p^{(\infty)}, J_p)\bigr\}\\ &= \bigl\{VPU :
    U,V\in\mathscr{B}(J_p)\bigr\},
  \end{align*}
  so that $\mathscr{W}(J_p)$ is the two-sided ideal
  of~$\mathscr{B}(J_p)$ generated by the single operator~$P$.
  On the other hand, Corollary~\ref{propWCopfg}, below, will show that
  $\mathscr{W}(J_p)$ is not finitely generated as a left ideal because
  $J_p$ is non-reflexive.
\end{example}

\section{Counting maximal left ideals}%
\label{section_countingmaxideals}
\noindent
Let~$E$ be an infinite-dimensional Banach space. An infinite family
$(E_\gamma)_{\gamma\in\Gamma}$ of non-zero, closed subspaces of~$E$ is
an \emph{unconditional Schauder decomposition} of~$E$ if, for each
$x\in E$, there is a unique family $(x_\gamma)_{\gamma\in\Gamma}$ with
$x_\gamma\in E_\gamma$ for each $\gamma\in\Gamma$ such that the series
$\sum_{\gamma\in\Gamma}x_\gamma$ converges unconditionally to~$x$.  In
this case we can associate a projection $P_\Upsilon\in\mathscr{B}(E)$
with each subset~$\Upsilon$ of~$\Gamma$ by the definitions
\begin{equation}\label{defnprojPA}
P_\emptyset = 0\quad \text{and}\quad P_\Upsilon x = \sum_{\gamma\in
  \Upsilon}x_\gamma\quad (x\in E)\quad \text{for}\quad
\Upsilon\neq\emptyset, \end{equation}
 where $(x_\gamma)_{\gamma\in\Gamma}$ is related to~$x$ as above.

Using this notion, we can transfer a classical algebraic result of
Rosenberg~\cite{Rosenberg} to~$\mathscr{B}(E)$. 

\begin{proposition}\label{countingleftideals} Let $E$ be a non-zero Banach space
  with an unconditional Schauder decomposition
  $(E_\gamma)_{\gamma\in\Gamma}$. Then the Banach algebra
  $\mathscr{B}(E)$ contains at least $2^{2^{|\Gamma|}}$ maximal left
   ideals which are not fixed.
\end{proposition}

\begin{proof}
  The power set~$\mathfrak{P}(\Gamma)$ of~$\Gamma$ is a Boolean
  algebra, and $$\mathfrak{I} = \bigl\{\Upsilon\in\mathfrak{P}(\Gamma) :
  |\Upsilon|<|\Gamma|\bigr\}$$ is a proper Boolean ideal
  of~$\mathfrak{P}(\Gamma)$.  Since~$\Gamma$ is infinite, a classical
  result of Posp\'{\i}\v{s}il (see~\cite{Pospisil}, or \cite[Corollary
  7.4]{cn} for an exposition) states that the
  collection~$\mathbb{M}_{\mathfrak{I}}$ of maximal Boolean ideals
  of~$\mathfrak{P}(\Gamma)$ containing~$\mathfrak{I}$ has
  cardinality~$2^{2^{|\Gamma|}}$.

  For each $\mathfrak{M}\in\mathbb{M}_{\mathfrak{I}}$, let
  $\mathsf{P}(\mathfrak{M}) = \{ P_\Upsilon :
  \Upsilon\in\mathfrak{M}\}\subseteq\mathscr{B}(E)$,
  where~$P_\Upsilon$ is the projection given
  by~\eqref{defnprojPA}. Assume towards a contradiction that the left
  ideal $\mathscr{L}_{\mathsf{P}(\mathfrak{M})}$ is not proper. Then,
  for some $n\in\N$, there are operators $T_1,\ldots,T_n\in
  \mathscr{B}(E)$ and sets $\Upsilon_1,\ldots,\Upsilon_n\in
  \mathfrak{M}$ such that $I_E = \sum_{j=1}^n T_j P_{\Upsilon_j}$.
  Right-composing both sides of this identity with the projection
  $P_{\Gamma\setminus \Upsilon}$, where $\Upsilon=\bigcup_{j=1}^n
  \Upsilon_j\in\mathfrak{M}$, we obtain $P_{\Gamma\setminus \Upsilon}
  = 0$, so that $\Gamma = \Upsilon\in\mathfrak{M}$, which contradicts
  the fact that~$\mathfrak{M}$ is a proper Boolean ideal.

  We can therefore choose a maximal left
  ideal~$\mathscr{M}_{\mathfrak{M}}$ of~$\mathscr{B}(E)$ such that
  $\mathscr{L}_{\mathsf{P}(\mathfrak{M})}\subseteq
  \mathscr{M}_{\mathfrak{M}}$.  This maximal left ideal
  $\mathscr{M}_{\mathfrak{M}}$ cannot be fixed because, for each $x\in
  E\setminus\{0\}$, we have $x =
  \sum_{\gamma\in\Gamma}P_{\{\gamma\}}x$, so that $P_{\{\gamma\}}x\neq
  0$ for some $\gamma\in\Gamma$. Hence $P_{\{\gamma\}}\notin
  \mathscr{M}\!\mathscr{L}_x$, but on the other hand
  $P_{\{\gamma\}}\in\mathscr{L}_{\mathsf{P}(\mathfrak{M})}\subseteq
  \mathscr{M}_{\mathfrak{M}}$ since
  $\{\gamma\}\in\mathfrak{I}\subseteq\mathfrak{M}$.

  Consequently we have a mapping $\mathfrak{M}\mapsto
  \mathscr{M}_{\mathfrak{M}}$ from~$\mathbb{M}_{\mathfrak{I}}$ into
  the set of non-fixed, maximal left ideals of~$\mathscr{B}(E)$.  We
  shall complete the proof by showing that this mapping is
  injective. Suppose that
  $\mathfrak{M},\mathfrak{N}\in\mathbb{M}_{\mathfrak{I}}$ are
  distinct, and take a set
  $\Upsilon\in\mathfrak{M}\setminus\mathfrak{N}$.  The maximality
  of~$\mathfrak{N}$ implies that $\Gamma\setminus
  \Upsilon\in\mathfrak{N}$, and therefore \[ I_E = P_\Upsilon +
  P_{\Gamma\setminus
    \Upsilon}\in\mathscr{L}_{\mathsf{P}(\mathfrak{M})} +
  \mathscr{L}_{\mathsf{P}(\mathfrak{N})}\subseteq
  \mathscr{M}_\mathfrak{M} + \mathscr{M}_\mathfrak{N}. \] Thus, since
  the left ideals $\mathscr{M}_\mathfrak{M}$ and
  $\mathscr{M}_\mathfrak{N}$ are proper, they are distinct.
\end{proof}

\begin{corollary}\label{wehaveanideal}
  Let $E$ be a non-zero Banach space with an unconditional Schauder
  decomp\-osition $(E_\gamma)_{\gamma\in \Gamma}$, and suppose that $E$
  contains a dense subset~$D$ such that $2^{|D|} < 2^{2^{|\Gamma|}}$.
  Then $\mathscr{B}(E)$ contains at least $2^{2^{|\Gamma|}}$ maximal
  left ideals which are not finitely generated.
\end{corollary}

\begin{proof} 
  Since each element of~$E$ is the limit point of a sequence in~$D$,
  we have $|E|\le|D|^{\aleph_0}$. Further, each operator on~$E$ is
  uniquely determined by its action on~$D$, and consequently
  \begin{equation}\label{cardBE} |\mathscr{B}(E)|\le |E^D| =
    |E|^{|D|}\le  \bigl(|D|^{\aleph_0}\bigr)^{|D|} = |D|^{|D|} =
    2^{|D|}, 
  \end{equation} 
  where the final equality follows from \cite[Lemma~5.6]{jech}, for
  example.  Hence $\mathscr{B}(E)$ contains at most
  $\bigl(2^{|D|}\bigr)^{\aleph_0} = 2^{|D|}$ countable subsets, so
  that $\mathscr{B}(E)$ contains at most $2^{|D|}$ countably-generated
  left ideals. On the other hand, Proposition~\ref{countingleftideals}
  implies that there are at least $2^{2^{|\Gamma|}}$ distinct maximal
  left ideals of~$\mathscr{B}(E)$.  We have $2^{|D|}<2^{2^{|\Gamma|}}$
  by the assumption, and hence $\mathscr{B}(E)$ contains at least
  $2^{2^{|\Gamma|}}$ maximal left ideals which are not countably
  generated, and thus not finitely generated.
\end{proof}

The most important case of this corollary is as follows.
\begin{corollary}\label{countablescauderdecomp}
  Let $E$ be a non-zero, separable Banach space with an unconditional Schau\-der
  de\-com\-po\-si\-tion $(E_\gamma)_{\gamma\in\Gamma}$.  Then
  $\mathscr{B}(E)$ contains precisely $2^{\,\mathfrak{c}}$ maximal
  left ideals which are not finitely generated. \end{corollary}

\begin{proof}
The index set~$\Gamma$ is necessarily countable because~$E$ is
separable. Hence, by Corollary~\ref{wehaveanideal}, $\mathscr{B}(E)$
contains at least~$2^{\,\mathfrak{c}}$ maximal left ideals which are
not finitely generated. On the other hand, \eqref{cardBE} implies that
$\mathscr{B}(E)$ has cardinality~$\mathfrak{c}$, so that
$\mathscr{B}(E)$ contains no more than $2^{\,\mathfrak{c}}$ distinct
subsets.
\end{proof}

\begin{example}\label{examplesofnumberofmaxleftideals} 
  \begin{romanenumerate} 
  \item\label{examplesofnumberofmaxleftideals1} Let $E$ be a Banach
    space with an unconditional Schau\-der basis~$(e_n)_{n\in\N}$.  Then
    $E$ satisfies the conditions of
    Corollary~\ref{countablescauderdecomp}, and hence $\mathscr{B}(E)$
    contains~$2^{\,\mathfrak{c}}$ maximal left ideals which are not
    finitely generated.

    The class of Banach spaces which have an unconditional Schauder
    basis is large and includes for instance the classical sequence
    spaces~$c_0$ and~$\ell_p$ for $p\in [1,\infty)$, the Lebesgue
    spaces~$L_p[0,1]$ for $p\in (1,\infty)$, the Lorentz and Orlicz
    sequence spaces~$d_{w,p}$ and~$h_M$ (\emph{e.g.}, see
    \cite[Chapter~4]{lt1}), the Tsirel\-son space~$T$ (\emph{e.g.},
    see \cite[Example~2.e.1]{lt1}), and the Schlum\-precht space~$S$
    (see \cite[Proposition 2]{Sch}).
  \item\label{examplesofnumberofmaxleftideals2} Suppose that $E$ is a
    Banach space   containing an infinite-di\-men\-sional, closed,
    complemented subspace~$F$ with  an unconditional Schauder
    decomposition $(F_\gamma)_{\gamma\in\Gamma}$. Then $E$ also has an
    unconditional Schauder decomposition, obtained by adding any
    closed, complementary subspace of~$F$ to the collection
    $(F_\gamma)_{\gamma\in\Gamma}$.

    In particular,
    generalizing~\romanref{examplesofnumberofmaxleftideals1}, we see
    that each separable Banach space~$E$ which contains an
    infinite-dimensional, closed, complemented subspace with an
    unconditional Schauder basis satisfies the conditions of
    Corollary~\ref{countablescauderdecomp}, and hence $\mathscr{B}(E)$
    contains $2^{\,\mathfrak{c}}$ maximal left ideals that are not
    finitely generated.  This applies for instance to \mbox{$E =
      L_1[0,1]$} because it contains a complemented copy of~$\ell_1$
    (\emph{e.g.}, see~\cite[Lemma~5.1.1]{ak}); to $E = C(K)$ for any
    infinite, compact metric space~$K$ because $E$ contains a
    complemented copy of~$c_0$ (\emph{e.g.},
    see~\cite[Proposition~4.3.11]{ak}); to $E = J_p$ for
    $p\in(1,\infty)$, the $p^{\text{th}}$ quasi-reflexive James space,
    because $J_p$ contains a complemented copy of~$\ell_p$
    (see~\cite[Lemma~6(d)]{em} or~\cite[Proposition~4.4(iii)]{lau1});
    and to $E = \mathscr{K}(X)$, where~$X$ is any Banach space with an
    unconditional Schauder basis, because~$E$ contains a complemented
    copy of~$c_0$ consisting of the compact operators whose matrix
    representation with respect to the unconditional Schauder basis is
    diagonal.
  \item\label{examplesofnumberofmaxleftideals3} There are separable
    Banach spaces~$E$ such that~$E$ has a un\-con\-di\-tion\-al
    Schauder decomposition $(E_n)_{n\in\N}$ with each~$E_n$
    finite-dimensional, but~$E$ does not have an un\-condi\-tional
    Schauder basis, notably Kalton and Peck's twisted
    $\ell_p$-spaces~$Z_p$ for $p\in(1,\infty)$ (see
    \cite[Corollary~9]{kaltonpeck} and the remark following it). Each
    such Banach space~$E$ satisfies the conditions of
    Corollary~\ref{countablescauderdecomp}, and hence $\mathscr{B}(E)$
    contains $2^{\,\mathfrak{c}}$ maximal left ideals which are not
    finitely generated.
  \end{romanenumerate}
\end{example}

\begin{remark} Corollary~\ref{countablescauderdecomp} is not true for
  all separable, infinite-dimen\-sional Banach spaces. Indeed, we shall
  show in Theorem~\ref{fewopsimpliesfixed}, below, that there are
  separable, infinite-dimensional Banach spaces~$E$ such that
  $\mathscr{B}(E)$ contains just one maximal left ideal which is not
  fixed, and this ideal is not finitely generated.
\end{remark}

\section{Proofs of the Dichotomy Theorems~\ref{dichotomythm}
  and~\ref{dich2}}%
\label{sectionLcontainsFE}
\noindent
The main purpose of this section is to prove
Theorems~\ref{dichotomythm} and~\ref{dich2}. We begin with the former,
whose proof is elementary.

\begin{proof}[Proof of Theorem~{\normalfont{\ref{dichotomythm}}}] 
  Let~$\mathscr{L}$ be a maximal left ideal of~$\mathscr{B}(E)$ such
  that $\mathscr{F}(E)\not\subseteq\mathscr{L}$.  Then there exist
  $x\in E$ and $\lambda\in E^*$ such that
  $x\otimes\lambda\notin\mathscr{L}$. The maximality of~$\mathscr{L}$
  implies that $I_E - T(x\otimes\lambda)\in\mathscr{L}$ for some
  operator \mbox{$T\in\mathscr{B}(E)$}, and $Tx\ne 0$
  because~$\mathscr{L}$ is proper. Choose $\mu\in E^*$ such that
  $\langle Tx,\mu\rangle = 1$, and define $P = I_E -
  Tx\otimes\mu\in\mathscr{B}(E)$. We then have $PTx = 0$, so that $P =
  P(I_E - T(x\otimes\lambda))\in\mathscr{L}$, and hence
  $\mathscr{M}\!\mathscr{L}_{Tx}\subseteq\mathscr{L}$ by
  Proposition~\ref{trivialfgmaxrightideals}%
  \romanref{trivialfgmaxrightideals1}.  Consequently these two maximal
  left ideals are equal, which shows that~$\mathscr{L}$ is fixed.
\end{proof}

Theorem~\ref{dichotomythm} easily leads to
the following stronger conclusion.

\begin{corollary}[Strong dichotomy for maximal left
  ideals]\label{June12dich1} Let~$E$ be a non-zero Banach space.
  Then, for each maximal left ideal~$\mathscr{L}$ of~$\mathscr{B}(E)$,
  exactly one of the following two alternatives holds:
  \begin{romanenumerate}
  \item\label{June12dich1i} $\mathscr{L}$ is fixed; or
  \item\label{June12dich1ii} $\mathscr{L}$ contains~$\mathscr{E}(E)$.
  \end{romanenumerate}
\end{corollary}

\begin{proof}
  Let~$\mathscr{L}$ be a non-fixed, maximal left ideal
  of~$\mathscr{B}(E)$. Then~$\mathscr{L}$ is closed and
  contains~$\mathscr{F}(E)$ by Theorem~\ref{dichotomythm}, so that
  $\overline{\mathscr{F}(E)}\subseteq\mathscr{L}$, and thus
  $\pi(\mathscr{L})$ is a maximal left ideal of the Calkin algebra
  $\mathscr{B}(E)\big/\,\overline{\mathscr{F}(E)}$, where
  \[ \pi\colon\mathscr{B}(E)\to
  \mathscr{B}(E)\big/\,\overline{\mathscr{F}(E)} \] de\-notes the
  quotient homomorphism. In particular $\pi(\mathscr{L})$ contains the
  Jacobson rad\-i\-cal
  of~$\mathscr{B}(E)\big/\,\overline{\mathscr{F}(E)}$, which implies
  that~$\mathscr{L}$ contains~$\mathscr{E}(E)$ by
  Remark~\ref{KleineckeDefnIness}, and therefore~\eqref{June12dich1ii}
  holds.
\end{proof}

\begin{remark}\label{inessin2sidedmaxideals}
  Corollary~\ref{June12dich1} can be seen as a counter-part for
  maximal left ideals of \cite[Proposition~6.6]{lau1}, which states
  that each maximal two-sided ideal of~$\mathscr{B}(E)$
  contains~$\mathscr{E}(E)$ for each
  infinite-di\-men\-sional Banach space~$E$.
\end{remark}

\begin{remark}
  Let~$\mathscr{A}$ be a unital $C^*$-algebra. We write $a\mapsto
  a^\star$ for the involution on~$\mathscr{A}$. (This should not be
  confused with the notation~$T^*$ for the adjoint of an operator $T$
  between Banach spaces used elsewhere in this paper.)  A \emph{state}
  on~$\mathscr{A}$ is a norm-one functional~$\lambda$ on~$\mathscr{A}$
  which is positive, in the sense that $\langle a^\star
  a,\lambda\rangle\ge 0$ for each $a\in\mathscr{A}$.  Given a state
  $\lambda$ on~$\mathscr{A}$, the set
  \begin{equation*}
    \mathscr{N}_\lambda = \bigl\{ a\in\mathscr{A} : \langle a^\star
    a,\lambda\rangle = 0\bigr\} \end{equation*} is a closed left ideal
  of~$\mathscr{A}$ by the Cauchy--Schwarz inequality (\emph{e.g.}, see
  \cite[Propo\-si\-tion~4.5.1]{kr1} or \cite[p.~93]{murphy}). The
  collection of all states on~$\mathscr{A}$ forms a weak$^*$-compact,
  convex subset of the dual space of~$\mathscr{A}$, called the
  \emph{state space} of~$\mathscr{A}$. Its extreme points are the
  \emph{pure states}
  on~$\mathscr{A}$. Prosser~\cite[Theorem~6.2]{prosser} has shown that
  the map $\lambda\mapsto\mathscr{N}_\lambda$ gives a bijective
  correspondence between the pure states on~$\mathscr{A}$ and the
  maximal left ideals of~$\mathscr{A}$; expositions of this result can
  be found in \cite[Theorem~10.2.10]{kr2} and
  \cite[Theorem~5.3.5]{murphy}.

  In the case where $\mathscr{A} = \mathscr{B}(H)$ for some Hilbert
  space~$H$, the fixed maximal left ideals correspond to the vector
  states, which are defined as follows.  Let $x\in H$ be a unit
  vector. Then the functional~$\omega_x$ given by
  \[ \langle T,\omega_x\rangle = (Tx\,|\,x)\quad
  (T\in\mathscr{B}(H)), \] where $(\,\cdot\, |\,\cdot\,)$ denotes the
  inner product on~$H$, is a pure state on~$\mathscr{B}(H)$, called
  the \emph{vector state} induced by~$x$; and we have
  $\mathscr{M}\!\mathscr{L}_x = \mathscr{N}_{\omega_x}$, as is easy to
  check. The conclusion of Corollary~\ref{June12dich1} is known in
  this case because $\mathscr{K}(H) = \mathscr{E}(H)$, and by
  \cite[Corollary~10.4.4]{kr2} each pure state~$\lambda$
  on~$\mathscr{B}(H)$ is either a vector state, or
  $\mathscr{K}(H)\subseteq\ker\lambda$, in which case
  $\mathscr{K}(H)\subseteq\mathscr{N}_\lambda$.

  Finally, suppose that the Hilbert space~$H$ is separable and
  infinite-dimensional. Then clearly $\mathscr{B}(H)$
  has~$\mathfrak{c}$ vector states, whereas it has
  $2^{\,\mathfrak{c}}$ pure states by
  \cite[Propo\-si\-tion~10.4.15]{kr2}. These conclusions also follow
  from Proposition~\ref{trivialfgmaxrightideals} and
  Example~\ref{examplesofnumberofmaxleftideals}%
  \romanref{examplesofnumberofmaxleftideals1}, respectively.
\end{remark}

We shall now turn our attention to the proof of Theorem~\ref{dich2}.
This requires some prepara\-tion.  Let~$E$ be a Banach space. For each
non-empty, bounded subset~$\Gamma$ of~$\mathscr{B}(E)$, we can define
an operator~$\Omega_\Gamma$ from the Banach space
\[ \ell_1(\Gamma,E^*) = \biggl\{ g\colon\Gamma\to E^* :
\sum_{T\in\Gamma} \| g(T)\|<\infty\biggr\} \] into~$E^*$ by
$\Omega_\Gamma g = \sum_{T\in\Gamma}T^*g(T)$ for each
$g\in\ell_1(\Gamma,E^*)$; that is,
\begin{equation}\label{defnOmegaGamma} 
  \langle x,\Omega_\Gamma g\rangle = 
  \sum_{T\in\Gamma}\bigl\langle Tx,g(T)\bigr\rangle\quad 
  (x\in E,\, g\in \ell_1(\Gamma,E^*)). 
\end{equation} 
The following lemma lists some basic properties of this operator, as
well as of the operator~$\Psi_\Gamma$ given
by~\eqref{defnPsiInfinite}, and explains their rele\-vance for our
present purpose. To state it, we require the notion of the
\emph{pre-annihila\-tor}~$\mbox{}^\perp M$ of a subset~$M$ of a dual
Banach space~$E^*$:
\[ \mbox{}^\perp M = \bigl\{ x\in E : \langle x, \lambda\rangle = 0\
   (\lambda\in M)\bigr\}. \]

\begin{lemma}\label{PsiInjLemma}
  Let $E$ be a non-zero Banach space, and let $\Gamma$ be a non-empty,
  bounded subset of~$\mathscr{B}(E)$. Then:
  \begin{romanenumerate}
  \item\label{PsiInjLemma2} $\ker\Psi_{\Gamma} =
    \mbox{}^\perp\Omega_\Gamma(\ell_1(\Gamma,E^*))$, and a
    non-zero element~$x$ of~$E$ belongs to this set if and only if
    $\mathscr{L}_{\Gamma}\subseteq\mathscr{M}\!\mathscr{L}_x;$
  \item\label{PsiInjLemma3} the following three conditions are
    equivalent:
    \begin{alphenumerate}
    \item\label{PsiInjLemma3a} no fixed maximal left ideal
      of~$\mathscr{B}(E)$ contains $\mathscr{L}_\Gamma;$
    \item\label{PsiInjLemma3c} the operator $\Psi_\Gamma$ is
      injective;
    \item\label{PsiInjLemma3b} the range of the
      operator~$\Omega_\Gamma$ is weak$^*$-dense in~$E^*$.
    \end{alphenumerate}
  \end{romanenumerate}
  Now suppose  either that the set~$\Gamma$ is finite or that the left ideal
  $\mathscr{L}_\Gamma$ is closed. Then:
  \begin{romanenumerate}\setcounter{smallromans}{2}
  \item\label{onedimleftideal2} for each $\lambda\in E^*$, the set
    $\mathscr{J}_\lambda = \{ y\otimes\lambda : y\in E\}$ is a left
    ideal of~$\mathscr{B}(E)$, and the following three conditions are
    equivalent:
    \begin{alphenumerate}
    \item\label{onedimleftideal2a}
      $\mathscr{J}_\lambda\subseteq\mathscr{L}_{\Gamma};$
    \item\label{onedimleftideal2b}
      $y\otimes\lambda\in\mathscr{L}_{\Gamma}$ for some $y\in
      E\setminus\{0\};$
    \item\label{onedimleftideal2c}
      $\lambda\in\Omega_{\Gamma}(\ell_1(\Gamma,E^*));$
    \end{alphenumerate}
  \item\label{onedimleftideal4} the operator $\Omega_\Gamma$ is
    surjective if and only if $\mathscr{L}_\Gamma$
    contains~$\mathscr{F}(E)$.
  \end{romanenumerate}
\end{lemma}

\begin{proof}   
  \romanref{PsiInjLemma2}. Suppose that $x\in\ker\Psi_\Gamma$. Then
  $Tx = 0$ for each $T\in\Gamma$, so that \eqref{defnOmegaGamma}
  implies that
  $x\in\mbox{}^\perp\Omega_\Gamma(\ell_1(\Gamma,E^*))$.

  Conversely, suppose that
  $x\in\mbox{}^\perp\Omega_\Gamma(\ell_1(\Gamma,E^*))$, and let
  $T\in\Gamma$ and $\lambda\in E^*$ be given. Defining
  $g\in\ell_1(\Gamma,E^*)$ by $g(T) = \lambda$ and $g(S) = 0$ for
  $S\in\Gamma\setminus\{T\}$, we have \mbox{$0 = \langle x, \Omega_\Gamma
  g\rangle = \langle Tx, \lambda\rangle$} by~\eqref{defnOmegaGamma}.
  Since $\lambda\in E^*$ was arbitrary, this shows that $Tx = 0$, and
  hence $x\in\ker\Psi_\Gamma$.

  To prove the second clause, we observe that for each $x\in
  E\setminus\{0\}$, we have $x\in\ker\Psi_{\Gamma}$ if and only if $
  \Gamma\subseteq\mathscr{M}\!\mathscr{L}_x$, and hence if and only if
  $ \mathscr{L}_{\Gamma}\subseteq \mathscr{M}\!\mathscr{L}_x$.

  \romanref{PsiInjLemma3}. The equivalence of~\alphref{PsiInjLemma3a}
  and~\alphref{PsiInjLemma3c} is immediate from (i).

  To see that~\alphref{PsiInjLemma3c} and~\alphref{PsiInjLemma3b} are
  equivalent, we observe that, by~\romanref{PsiInjLemma2} and
  \cite[Proposition~2.6.6(c)]{meg}, the weak$^*$ closure of the range
  of~$\Omega_\Gamma$ is equal to the annihilator
  \[ (\ker\Psi_\Gamma)^\perp = \bigl\{\lambda\in E^* : \langle
  x,\lambda\rangle = 0\ (x\in\ker\Psi_\Gamma)\bigr\} \]
  of~$\ker\Psi_\Gamma$. Hence~\alphref{PsiInjLemma3c}
  implies~\alphref{PsiInjLemma3b}. Conversely, suppose that
  $(\ker\Psi_\Gamma)^\perp = E^*$. Then
  \cite[Proposition~1.10.15(c)]{meg} implies that $\ker\Psi_\Gamma =
  \mbox{}^\perp(E^*) = \{0\}$.
   
  \romanref{onedimleftideal2}.  Equation~\eqref{eqCompRankOneOp} shows
  that $\mathscr{J}_\lambda$ is a left ideal.

  The implication
  \alphref{onedimleftideal2a}$\Rightarrow$\alphref{onedimleftideal2b}
  is evident.

  \alphref{onedimleftideal2b}$\Rightarrow$\alphref{onedimleftideal2c}.
  Suppose that $y\otimes\lambda\in\mathscr{L}_{\Gamma}$ for some $y\in
  E\setminus\{0\}$. Then there are $n\in\N$,
  $S_1,\ldots,S_n\in\mathscr{B}(E)$, and $T_1,\ldots,T_n\in\Gamma$
  such that $y\otimes\lambda = \sum_{j=1}^n S_jT_j$, where we may
  suppose that $T_1,\ldots,T_n$ are distinct. Choose $\mu\in E^*$ such
  that $\langle y,\mu\rangle =1$, and define $g\colon\Gamma\to E^*$ by
  \[ g(T) = \begin{cases} S^*_j\mu &\text{if}\ T=T_j\ \text{for some}\
    j\in\{1,\ldots,n\},\\ 0 &\text{otherwise.} \end{cases} \] Then $g$
  has finite support, so that trivially it belongs
  to~$\ell_1(\Gamma,E^*)$, and $\Omega_\Gamma g = \lambda$
  because~\eqref{defnOmegaGamma} implies that
  \[ \langle x,\Omega_\Gamma g\rangle = \sum_{j=1}^n\langle T_jx,
  S_j^*\mu\rangle = \biggl\langle \sum_{j=1}^n S_jT_j
  x,\mu\biggr\rangle = \bigl\langle (y\otimes\lambda)x,\mu\bigr\rangle
  = \langle x,\lambda\rangle \] for each $x\in E$. Hence
  $\lambda\in\Omega_{\Gamma}(\ell_1(\Gamma,E^*))$.

  \alphref{onedimleftideal2c}$\Rightarrow$\alphref{onedimleftideal2a}.
  Suppose that $\lambda = \Omega_{\Gamma}g$ for some
  $g\in\ell_1(\Gamma,E^*)$. Then, for each $y\in E$, we have
  \[ y\otimes \lambda = \sum_{T\in\Gamma}y\otimes T^*g(T) =
  \sum_{T\in\Gamma}(y\otimes g(T))T, \] which belongs
  to~$\mathscr{L}_\Gamma$ because each term of the sum on the
  right-hand side does, and either this sum is finite
  or~$\mathscr{L}_\Gamma$ is closed by the assumption.

  \romanref{onedimleftideal4}.  Suppose that $\Omega_{\Gamma}$ is
  surjective. Then~\romanref{onedimleftideal2} implies that
  $\mathscr{J}_\lambda\subseteq\mathscr{L}_{\Gamma}$ for each
  $\lambda\in E^*$, and consequently $\mathscr{F}(E) =
  \operatorname{span}\bigcup_{\lambda\in
    E^*}\mathscr{J}_\lambda\subseteq\mathscr{L}_{\Gamma}$.
  
  Conversely, suppose that
  $\mathscr{F}(E)\subseteq\mathscr{L}_{\Gamma}$, and let $\lambda\in
  E^*$ be given. Since
  \mbox{$\mathscr{J}_\lambda\subseteq\mathscr{F}(E)$},
  \romanref{onedimleftideal2}~implies that
  $\lambda\in\Omega_{\Gamma}(\ell_1(\Gamma,E^*))$, so that the map
  $\Omega_{\Gamma}$ is surjective.
\end{proof}

\begin{corollary}\label{Omegaclosedrange}
  Let $E$ be a Banach space, and let $\Gamma$ be a non-empty, bounded
  subset of~$\mathscr{B}(E)$ such that the left
  ideal~$\mathscr{L}_\Gamma$ is closed. Then the
  operator~$\Omega_{\Gamma}$ has closed range.
\end{corollary}
\begin{proof}
  We may suppose that $E$ is non-zero.  Consider a sequence
  $(\lambda_j)_{j\in\N}$ in~$\Omega_{\Gamma}(\ell_1(\Gamma,E^*))$ that
  converges to some element $\lambda\in E^*$.
  Lemma~\ref{PsiInjLemma}\romanref{onedimleftideal2} shows that
  $\mathscr{J}_{\lambda_j}\subseteq\mathscr{L}_{\Gamma}$ for each
  $j\in\N$, and hence
  \[ y\otimes\lambda = \lim_{j\to\infty}
  y\otimes\lambda_j\in\mathscr{L}_{\Gamma}\quad (y\in E) \] because
  $\mathscr{L}_{\Gamma}$ is closed, so that
  $\lambda\in\Omega_{\Gamma}(\ell_1(\Gamma,E^*))$ by another
  application of Lemma~\ref{PsiInjLemma}\romanref{onedimleftideal2}.
  Thus  $\Omega_{\Gamma}$ has closed range.
\end{proof}

We can now characterize the closed left ideals of~$\mathscr{B}(E)$
that contain~$\mathscr{F}(E)$ as follows, provided either  that $E$ is
reflexive or that we restrict our attention to the closed left ideals that
are finitely generated. Note that Theorem~\ref{dich2} is simply a
restatement of the equivalence of
conditions~\alphref{charfixedideals1}
and~\alphref{charfixedideals2.5}.

\begin{theorem}\label{charfixedideals}
  Let~$E$ be a non-zero Banach space, let~$\mathscr{L}$ be a closed
  left ideal of~$\mathscr{B}(E)$, and take a non-empty, bounded
  subset~$\Gamma$ of~$\mathscr{B}(E)$ such that $\mathscr{L} =
  \mathscr{L}_\Gamma$.  Suppose either that $E$ is reflexive
  or that~$\Gamma$ is finite. Then the following six conditions are
  equivalent:
  \begin{alphenumerate}
  \item\label{charfixedideals1} no fixed maximal left ideal
    of~$\mathscr{B}(E)$ contains $\mathscr{L};$
  \item\label{charfixedideals3} the operator $\Psi_{\Gamma}$ is
    injective;
  \item\label{charfixedideals4} the operator $\Psi_{\Gamma}$ is
    bounded below;
  \item\label{charfixedideals5} the range of the
    operator~$\Omega_\Gamma$ is weak$^*$-dense in~$E^*;$
  \item\label{charfixedideals2} the operator $\Omega_\Gamma$ is
    surjective; 
  \item\label{charfixedideals2.5} $\mathscr{L}$ contains
    $\mathscr{F}(E)$.
  \end{alphenumerate}
\end{theorem}

\begin{proof}
  The following implications hold without supposing that~$E$ is
  reflexive or~$\Gamma$ is finite: Conditions
  \alphref{charfixedideals1}, \alphref{charfixedideals3},
  and~\alphref{charfixedideals5} are mutually equivalent by
  Lemma~\ref{PsiInjLemma}\romanref{PsiInjLemma3}, while
  conditions~\alphref{charfixedideals2}
  and~\alphref{charfixedideals2.5} are equivalent by
  Lemma~\ref{PsiInjLemma}\romanref{onedimleftideal4}.  Evidently
  \alphref{charfixedideals4} implies~\alphref{charfixedideals3}, and
  \alphref{charfixedideals2} im\-plies~\alphref{charfixedideals5}. In
  fact, \alphref{charfixedideals2} implies \alphref{charfixedideals4},
  as we shall now show. Suppose that~$\Omega_\Gamma$ is surjective. We
  can define a linear isometry
  \mbox{$\Xi_\Gamma\colon\ell_1(\Gamma,E^*)\to\ell_\infty(\Gamma,E)^*$}
  by
  \begin{equation}\label{defnUpsilonoperator}
  \langle f,\Xi_\Gamma g\rangle = \sum_{T\in\Gamma} \bigl\langle f(T),
  g(T)\bigr\rangle\quad (f\in \ell_\infty(\Gamma,E),\,
  g\in\ell_1(\Gamma,E^*)), \end{equation} and $\Omega_\Gamma =
  \Psi_\Gamma^*\Xi_\Gamma$. Hence the surjectivity of~$\Omega_\Gamma$
  implies that~$\Psi_\Gamma^*$ is surjective, and
  therefore~$\Psi_\Gamma$ is bounded below by~\eqref{dualitySurBB}.

  The remaining implications do require further assumptions.  We
  consider first the case where $E$ is reflexive. Then the weak and
  weak$^*$ topologies on~$E^*$ coincide, so that the range
  of~$\Omega_\Gamma$ is weak$^*$-dense if and only if it is weakly
  dense, if and only if it is norm-dense by Mazur's theorem
  (\emph{e.g.}, see \cite[Theorem~2.5.16]{meg}). However,
  $\Omega_\Gamma$ has closed range by
  Cor\-ol\-lary~\ref{Omegaclosedrange}, and
  consequently~\alphref{charfixedideals5}
  implies~\alphref{charfixedideals2}, which completes the proof in
  this case.

  Secondly, suppose that the set $\Gamma$ is finite. Then the isometry
  $$\Xi_\Gamma\colon \ell_1(\Gamma,E^*)\to\ell_\infty(\Gamma,E)^*$$
  defined by~\eqref{defnUpsilonoperator}, above, is an isomorphism, so
  that~$\Omega_\Gamma$ and~$\Psi_\Gamma^*$ are equal up to an
  isometric identification. Hence~\alphref{charfixedideals4}
  and~\alphref{charfixedideals2} are equivalent
  by~\eqref{dualitySurBB}. Moreover,
  Cor\-ol\-lary~\ref{Omegaclosedrange} shows that $\Psi_\Gamma^*$ has
  closed range, and therefore~$\Psi_{\Gamma}$ has closed range by the
  closed range theorem (\emph{e.g.}, see
  \cite[Theorem~3.1.21]{meg}). Thus~\alphref{charfixedideals3}
  implies~\alphref{charfixedideals4}, and the proof is complete.
\end{proof}
For later reference, we note that the arguments which establish the
equivalence of conditions~\alphref{charfixedideals4},
\alphref{charfixedideals2}, and~\alphref{charfixedideals2.5} for
finite~$\Gamma$ given in the first and last paragraph of the proof of
Theorem~\ref{charfixedideals}, above, remain true in the case where
the left ideal $\mathscr{L} = \mathscr{L}_\Gamma$ is not necessarily
closed. Hence we have the following conclusion.

\begin{corollary}\label{Sept2011keylemma}
  Let~$E$ be a Banach space, and let~$\Gamma$ be a non-empty, finite
  subset of~$\mathscr{B}(E)$. Then the following three conditions are
  equivalent:
  \begin{alphenumerate}
  \item the operator $\Psi_{\Gamma}$ is bounded below;
  \item the operator $\Omega_\Gamma$ is surjective;
  \item $\mathscr{L}_\Gamma$ contains $\mathscr{F}(E)$.
  \end{alphenumerate}
\end{corollary}

As an easy consequence of these results, we obtain that the ideal of
weakly compact operators is finitely generated as a left ideal only in
the trivial case where it is not proper.
\begin{corollary}\label{propWCopfg}
  The following three conditions are equivalent for a Banach
  space~$E\colon$
  \begin{alphenumerate}
  \item\label{propWCopfg1} $\mathscr{W}(E)$ is finitely generated as a
    left ideal;
  \item\label{propWCopfg3} $\mathscr{W}(E) = \mathscr{B}(E);$
   \item\label{propWCopfg2} the Banach space $E$ is reflexive.
   \end{alphenumerate}
\end{corollary}

\begin{proof} The equivalence of~\alphref{propWCopfg3}
  and~\alphref{propWCopfg2} is standard, and \alphref{propWCopfg3}
  obviously implies~\alphref{propWCopfg1}.

  To see that~\alphref{propWCopfg1} implies~\alphref{propWCopfg2},
  suppose that $\mathscr{W}(E) = \mathscr{L}_{\Gamma}$ for some
  non-empty, finite subset $\Gamma = \{T_1,\ldots,T_n\}$
  of~$\mathscr{B}(E)$. It clearly suffices to consider the case where
  $E$ is non-zero.  Theorem~\ref{charfixedideals} (or
  Corollary~\ref{Sept2011keylemma}) implies that the operator
  $\Psi_{\Gamma}$ is bounded below because $\mathscr{W}(E)$
  contains~$\mathscr{F}(E)$. Since the operators $T_1,\ldots,T_n$ are
  weakly compact, the same is true for~$\Psi_{\Gamma}$ by the
  definition~\eqref{defnOpPsi}. Hence the
  Davis--Figiel--Johnson--Pe{\l}czy{\a'n}ski factorization theorem
  (see~\cite{dfjp}, or~\cite[Theorem~2.g.11]{lt2} for an exposition)
  implies that, for some reflexive Banach space~$F$, there are
  operators $R\in\mathscr{B}(E,F)$ and $S\in\mathscr{B}(F,E^n)$ such
  that $\Psi_{\Gamma} = SR$. Now $R$ is bounded below
  because~$\Psi_{\Gamma}$ is, and there\-fore $E$ is isomorphic to the
  subspace $R(E)$ of the reflexive space~$F$, so that $E$ is
  reflexive. \end{proof}

We conclude this section with an example that shows that
Theorem~\ref{charfixedideals} may not be true if we drop the
assumption that either the Banach space~$E$ is reflexive or the
set~$\Gamma$ is finite. This requires the following easy variant of
Lemma~\ref{PsiInjLemma}\romanref{onedimleftideal2}.

\begin{lemma}\label{lemma5712}
  Let $T$ be an operator on a Banach space~$E$, and suppose that
  $y\otimes\lambda\in\overline{\mathscr{L}_{\{T\}}}$ for some $y\in
  E\setminus\{0\}$ and $\lambda\in E^*$. Then
  $\lambda\in\overline{T^*(E^*)}$.
\end{lemma}

\begin{proof}
  Let $(S_j)_{j\in\N}$ be a sequence in~$\mathscr{B}(E)$ such that
  $S_jT\to y\otimes\lambda$ as $j\to\infty$, and choose $\mu\in E^*$
  such that $\langle y,\mu\rangle = 1$.  Then
  \[ T^*(S_j^*\mu) = (S_jT)^*\mu\to (y\otimes\lambda)^*\mu = \langle
  y,\mu\rangle\lambda = \lambda\quad\text{as}\quad j\to\infty, \]
  from which the conclusion follows.
\end{proof}

\begin{example}\label{c0exampleFailureof2ndDichotomy}
  Let $T$ be the operator on~$\ell_\infty$ given by
  \begin{equation}\label{c0exampleFailureof2ndDichotomyEq1}
    T(\alpha_j)_{j\in\N} = \biggl(-\frac{\alpha_n}{2^{n}} +
    \sum_{j=n+1}^\infty \frac{\alpha_j}{2^{j}}\biggr)_{n\in\N}\quad
    ((\alpha_j)_{j\in\N}\in \ell_\infty). \end{equation} Then~$T$ is
  compact and leaves the subspace~$c_0$ invariant. Define $$T_0\colon
  x\mapsto Tx,\quad c_0\to c_0,$$ and consider the closed left ideal
  $\mathscr{L} = \overline{\mathscr{L}_{\{T_0\}}}$
  of~$\mathscr{B}(c_0)$. We have
  $\mathscr{L}\subseteq\mathscr{K}(c_0)$ because~$T_0$ is compact. Our
  aim is to show that~$\mathscr{L}$ satisfies
  condition~\alphref{charfixedideals1}, but not
  condition~\alphref{charfixedideals2.5}, of
  Theorem~\ref{charfixedideals}.

  We begin by verifying that $\ker T = \C(1,1,\ldots)$. First, it  is clear
  that $T(1,1,\ldots) = (0,0,\ldots)$. Conversely, suppose that
  $(\alpha_j)_{j\in\N}\in\ker T$. Then
  \[ \frac{\alpha_n}{2^{n}} = \sum_{j=n+1}^\infty
  \frac{\alpha_j}{2^{j}}\quad (n\in\N), \] so that
  \[ \frac{\alpha_n}{2^{n}} = \frac{\alpha_{n+1}}{2^{n+1}} +
  \sum_{j=n+2}^\infty \frac{\alpha_j}{2^{j}} =
  \frac{\alpha_{n+1}}{2^{n+1}} + \frac{\alpha_{n+1}}{2^{n+1}} =
  \frac{\alpha_{n+1}}{2^{n}}\quad (n\in\N). \] Hence $\alpha_n =
  \alpha_{n+1}$ for each $n\in\N$, and the conclusion follows.

  This shows in particular that $T_0$ is injective because
  $c_0\cap\ker T = \{0\}$. Consequently
  $T_0\notin\mathscr{M}\!\mathscr{L}_x$ for each $x\in
  c_0\setminus\{0\}$, and so $\mathscr{L}$ satisfies
  condition~\alphref{charfixedideals1} of
  Theorem~\ref{charfixedideals}.

  On the other hand, identifying $c_0^{**}$ with~$\ell_\infty$ in the
  usual way, we find that $T_0^{**} = T$, which is not injective, so
  that $T_0^*$ does not have norm-dense range
  by~\cite[Theorem~3.1.17(b)]{meg}. Take $\lambda\in
  c_0^*\setminus\overline{T_0^*(c_0^*)}$ and $y\in
  E\setminus\{0\}$. Then, by Lemma~\ref{lemma5712},
  $y\otimes\lambda\notin\mathscr{L}$, so that $\mathscr{L}$ does not
  satisfy condition~\alphref{charfixedideals2.5} of
  Theorem~\ref{charfixedideals}.
\end{example}

\section{`Classical' Banach spaces for which each finitely-generated,
  maximal left ideal is fixed}\label{sectionAllfgidealsfixed}
\noindent
The purpose of this section is to show that
Question~\romanref{DalesQ1} has a positive answer for many standard
Banach spaces~$E$.

We begin by showing that a much stronger conclusion is true in certain
cases, namely that no finitely-generated, proper left ideal
of~$\mathscr{B}(E)$ contains~$\mathscr{F}(E)$. This result relies on
the following characterization of the finite subsets~$\Gamma$ of
$\mathscr{B}(E)$ that do \textsl{not} generate a proper left ideal in
terms of standard operator-theoretic properties of~$\Psi_\Gamma$.

\begin{lemma}\label{properidealgeneratedlemma}
  Let~$E$ be a non-zero Banach space. Then the following three
  conditions are equivalent for each non-empty, finite subset~$\Gamma$
  of~$\mathscr{B}(E)\colon$
  \begin{alphenumerate}
  \item\label{properidealgeneratedlemma1} the operator $\Psi_{\Gamma}$
    is bounded below and its range is complemented in~$E^{|\Gamma|};$
  \item\label{properidealgeneratedlemma2} the operator
    $\Psi_{\Gamma}$ is left invertible;
  \item\label{properidealgeneratedlemma3} $\mathscr{L}_{\Gamma} =
    \mathscr{B}(E)$.
  \end{alphenumerate}
\end{lemma}

\begin{proof} The equivalence of~\alphref{properidealgeneratedlemma1}
  and~\alphref{properidealgeneratedlemma2} is an easy standard result,
  true for any operator be\-tween Banach spaces, while the equivalence
  of \alphref{properidealgeneratedlemma2}
  and~\alphref{properidealgeneratedlemma3} follows immediately
  from~\eqref{eqFGLeftideal04072012}.
\end{proof}

\begin{proposition}\label{charFEinFGleftidel} Let $E$ be a 
  Banach space, and let $n\in\N$.  Then $\mathscr{F}(E)$ is contained
  in a proper left ideal of~$\mathscr{B}(E)$ generated by~$n$
  operators if and only if $E^n$ contains a closed sub\-space which is
  isomorphic to~$E$ and which is not complemented in~$E^n$.
\end{proposition}

\begin{proof} We may suppose that $E$ is non-zero, and prove both
  implications by contraposition. 

  $\Rightarrow$. Suppose that every closed subspace of~$E^n$ that is
  isomorphic to~$E$ is complemented in~$E^n$, and let~$\Gamma$ be a
  subset of~$\mathscr{B}(E)$ of cardinality~$n$ such that
  $\mathscr{F}(E)\subseteq \mathscr{L}_{\Gamma}$.  We must prove that
  $\mathscr{L}_{\Gamma} = \mathscr{B}(E)$; that is, by
  Lemma~\ref{properidealgeneratedlemma}, we must show that the
  operator~$\Psi_{\Gamma}$ is bounded below and has complemented
  range. Corollary~\ref{Sept2011keylemma} implies that $\Psi_{\Gamma}$
  is indeed bounded below, and its range is therefore a closed
  subspace of~$E^n$ isomorphic to~$E$, so that it is complemented by
  the assumption.

  $\Leftarrow$. Suppose that $\mathscr{B}(E)$ is the only left ideal
  with (at most)~$n$ generators that contains~$\mathscr{F}(E)$, and
  let~$F$ be a closed subspace of~$E^n$ such that~$F$ is isomorphic
  to~$E$.  We must prove that $F$ is complemented in~$E^n$. Take an
  operator $T\in\mathscr{B}(E, E^n)$ which is bounded below and has
  range~$F$, and let $T_j = \rho_j T\in\mathscr{B}(E)$, where
  \[ \rho_j\colon\ (x_k)_{k=1}^n\mapsto x_j,\quad E^n\to E\quad
  (j\in\{1,\ldots,n\}). \] Re-ordering the coordinates of~$E^n$, we
  may suppose that there exists a number \mbox{$m\le n$} such that
  $T_1,\ldots,T_m$ are distinct and the set $\Gamma =
  \{T_1,\ldots,T_m\}$ contains~$T_j$ for
  each~$j\in\{1,\ldots,n\}$. Then \mbox{$\|\Psi_{\Gamma}x\| =
    \max_{1\le j\le m}\|T_jx\| = \|Tx\|$} for each $x\in E$, so
  that~$\Psi_\Gamma$ is bounded below, and there\-fore
  $\mathscr{L}_{\Gamma}$ contains~$\mathscr{F}(E)$ by
  Corollary~\ref{Sept2011keylemma}. Now the assump\-tion implies that
  \mbox{$\mathscr{L}_{\Gamma} = \mathscr{B}(E)$}; that is, we can find
  an operator $S\in\mathscr{B}(E^m,E)$ such that
  \[ I_E = S\Psi_\Gamma = S\biggl(\sum_{j=1}^m
  \iota_j\rho_j\biggr)T. \] Hence~$T$ has a left inverse,
  and consequently its range, which is equal to~$F$, is complemented
  in~$E^n$.  \end{proof}

Combining this result with Theorem~\ref{dichotomythm}, we reach the
following conclusion.

\begin{corollary}\label{noproperfgleftidealcontainsFE}
  Let $E$ be a non-zero Banach space such that, for each $n\in\N$,
  every closed subspace of~$E^n$ that is isomorphic to~$E$ is
  complemented in~$E^n$.  Then~$\mathscr{B}(E)$ is the only
  finitely-generated left ideal of~$\mathscr{B}(E)$ which
  contains~$\mathscr{F}(E)$, and hence each finitely-generated,
  maximal left ideal of~$\mathscr{B}(E)$ is fixed.
\end{corollary}

\begin{example}\label{exampleAutocomplemented} 
The condition of Corollary~\ref{noproperfgleftidealcontainsFE} on the
Banach space~$E$ is satisfied in each of the following three cases:
  \begin{romanenumerate}
  \item $E$ is a Hilbert space.
  \item\label{exampleAutocomplemented3} $E$ is an injective Banach
    space; that is, whenever a Banach space~$F$ contains a closed
    sub\-space~$G$ which is isomorphic to~$E$, then $G$ is
    complemented in~$F$. For instance, the Banach space $E =
    \ell_\infty(\mathbb{I})$ is injective for each non-empty
    set~$\mathbb{I}$. More generally, $ C(K)$ is injective whenever
    the Hausdorff space~$K$ is Stonean (that is, compact and extremely
    disconnected), as shown by Goodner~\cite{goodner} and
    Nachbin~\cite{nachbin} for real scalars and generalized to the
    complex case by Cohen~\cite{cohen}.
  \item\label{exampleAutocomplementedc0} $E = c_0(\mathbb{I})$ for a
    non-empty set~$\mathbb{I}$ (this follows from Sobczyk's
    theorem~\cite{sob} for countable~$\mathbb{I}$ and
    from~\cite{granero} (or~\cite[Proposition~2.8]{argyrosetal}) in
    the general case); here $c_0(\mathbb{I})$ denotes the closed
    subspace of~$\ell_\infty(\mathbb{I})$ consisting of those
    functions~$f\colon\mathbb{I}\to\C$ for which the set $\bigl\{
    i\in\mathbb{I} : |f(i)|\ge\epsilon\bigr\}$ is finite for each
    $\epsilon>0$.
  \end{romanenumerate}
  Thus, in each of these three cases, $\mathscr{B}(E)$ is the only
  finitely-generated left ideal of~$\mathscr{B}(E)$ which
  contains~$\mathscr{F}(E)$, and each finitely-generated, maximal left
  ideal of~$\mathscr{B}(E)$ is fixed.
\end{example}

Our next goal is to prove a result (Theorem~\ref{Thmlp}) which, under
much less restrictive con\-di\-tions on the Banach space~$E$ than
Corollary~\ref{noproperfgleftidealcontainsFE}, gives the slightly
weaker con\-clu\-sion that $\mathscr{B}(E)$ is the only
finitely-generated left ideal of~$\mathscr{B}(E)$ which
contains~$\mathscr{K}(E)$. We note in particular that
Corollary~\ref{June12dich1} ensures that this conclusion is still
strong enough to ensure that each finitely-generated, maximal left
ideal of~$\mathscr{B}(E)$ is fixed, thus answering
Question~\romanref{DalesQ1} positively for a large number of Banach
spaces.

Let~$E$ be a Banach space with a Schauder basis $\mathbf{e} =
(e_j)_{j\in\N}$. For each $k\in\N$, we denote by $P_k$ the
$k^{\text{th}}$ basis projection associated with~$\mathbf{e}$. The
\emph{basis constant} of~$\mathbf{e}$ is \[ K_{\mathbf{e}} =
\sup\bigl\{\| P_k\| : k\in\N\bigr\}\in [1,\infty). \] The
  basis~$\mathbf{e}$ is \emph{monotone} if $K_{\mathbf{e}} = 1$.
\begin{lemma}\label{fdposdiagopnorm}
  Let $E$ be a Banach space with a Schauder basis $\mathbf{e} =
  (e_j)_{j\in\N}$, and let $\gamma = (\gamma_j)_{j\in\N}$ be a
  decreasing sequence of non-negative real numbers. Then
  \begin{equation}\label{fdposdiagopnormEq1}
    \Delta_\gamma\colon\ \sum_{j=1}^\infty\alpha_j e_j\mapsto
    \sum_{j=1}^\infty\gamma_j\alpha_j e_j
  \end{equation}
  defines an operator~$\Delta_\gamma$ on~$E$ of norm at most
  $K_{\mathbf{e}}\gamma_1$. This operator is compact if and only if
  $\gamma_j\to 0$ as $j\to\infty$.
\end{lemma}
\begin{proof} 
  Equation~\eqref{fdposdiagopnormEq1} clearly defines a linear
  mapping~$\Delta_\gamma$ from the dense subspace
  $\operatorname{span}\{e_j : j\in\N\}$ of~$E$ into~$E$, and so it
  suffices to show that this mapping is bounded with norm at most
  $K_{\mathbf{e}}\gamma_1$. Now, for each element $x =
  \sum_{j=1}^k\alpha_je_j$ of~$\operatorname{span}\{e_j :
  j\in\N\}$, where $k\in\N$ and $\alpha_1,\ldots,\alpha_k\in\C$, we
  have
  \[ \Delta_\gamma x = \gamma_1P_1x + \sum_{j=2}^k\gamma_j(P_jx -
  P_{j-1}x) = \sum_{j=1}^{k-1}(\gamma_j - \gamma_{j+1})P_jx +
  \gamma_kx, \] and thus
  \[ \|\Delta_\gamma x\|\le \sum_{j=1}^{k-1}(\gamma_j - \gamma_{j+1})
  K_{\mathbf{e}}\|x\| + \gamma_k\|x\|\le
  K_{\mathbf{e}}\gamma_1\|x\|, \] as required.

  To prove the final clause, we note that, by a standard result,
  $(P_j)_{j\in\N}$ is a bounded left approximate identity
  for~$\mathscr{K}(E)$ (\emph{e.g.}, see~\cite[p.~318]{dales}), so
  that~$\Delta_\gamma$ is compact if and only if
  $P_j\Delta_\gamma\to\Delta_\gamma$ as $j\to\infty$. Hence the
  estimates
  \[ \gamma_{j+1}\le\bigl\|(I_E - P_j)\Delta_\gamma\bigr\|\le
  K_{\mathbf{e}}(K_{\mathbf{e}}+1)\gamma_{j+1}\quad (j\in\N), \]
  which are easy to verify, give the result.
\end{proof}

\begin{corollary}\label{posdiagopinfdim}
  Let $E$ be a Banach space with a Schauder basis $\mathbf{e} =
  (e_j)_{j\in\N}$, let $k\in\N$, and let $\eta = (\eta_j)_{j=1}^k$ be
  an increasing $k$-tuple of non-negative real numbers. Then
  \begin{equation}\label{posdiagopinfdimEq1}
    \Theta_\eta\colon\ \sum_{j=1}^\infty\alpha_j e_j\mapsto
    \sum_{j=1}^k\eta_j\alpha_j e_j \end{equation}
  defines an operator on~$E$ of norm at most~$2K_{\mathbf{e}}\eta_k$.
\end{corollary}
\begin{proof}
  Define \[ f_j = \begin{cases} e_{k-j+1} &\text{for}\ j\le k,\\ e_j
    &\text{for}\ j>k.\end{cases} \] Then $\mathbf{f} = (f_j)_{j\in\N}$
  is a Schauder basis for~$E$ (because we have re-ordered only finitely
  many vectors of the original basis~$\mathbf{e}$), and the
  $m^{\text{th}}$ basis projection associated with $\mathbf{f}$ is
  given by $P_k - P_{k-m}$ for $m < k$ and $P_m$ for $m\ge k$, so that
  $K_{\mathbf{f}}\le 2K_{\mathbf{e}}$. Now Lemma~\ref{fdposdiagopnorm}
  gives the desired conclusion because \[ \Theta_\eta =
  \Delta_\gamma\colon\ \sum_{j=1}^\infty \alpha_j
  f_j\mapsto\sum_{j=1}^\infty \gamma_j\alpha_j f_j,\quad E\to E, \]
  where $\gamma=(\gamma_j)_{j\in\N}$ denotes the decreasing sequence
  $(\eta_k, \eta_{k-1},\ldots,\eta_1,0,0,\ldots)$.
\end{proof}

We now come to our key lemma. 
\begin{lemma}\label{TKTKNJLthm11June2012}
  Let $E$ be a Banach space with a monotone Schauder
  basis, and let $\Gamma$ be a~non-empty, finite subset
  of~$\mathscr{B}(E)$ for which
  $\mathscr{F}(E)\subseteq\mathscr{L}_{\Gamma}$.  Then the sequence
  $(t_j)_{j\in\N}$ given by
  \begin{equation}\label{TKTKNJLthm11June2012Eq1}
    t_j = \inf\bigl\{\| T\| : T\in\mathscr{B}(E^{|\Gamma|},E),\, P_j =
    T\Psi_{\Gamma}\bigr\}\in(0,\infty) \quad
    (j\in\N) \end{equation} 
  is increasing.
  
  Suppose that $(t_j)_{j\in\N}$ is unbounded, and let $\gamma =
  (t_j^{-1/2})_{j\in\N}$. Then the operator~$\Delta_\gamma$ given
  by~\eqref{fdposdiagopnormEq1} is compact and does not belong
  to~$\mathscr{L}_{\Gamma}$. 
\end{lemma}
\begin{proof}
  Set $n=|\Gamma|\in\N$. For each $j\in\N$, we have
  $P_j\in\mathscr{F}(E)\subseteq\mathscr{L}_{\Gamma}$, so
  that~\eqref{eqFGLeftideal04072012} ensures that the set appearing in
  the definition~\eqref{TKTKNJLthm11June2012Eq1} of~$t_j$ is
  non-empty, and further that $t_j\ge\|\Psi_{\Gamma}\|^{-1}>0$.  To see that
  $t_{j+1}\ge t_j$, suppose that $P_{j+1} = T\Psi_{\Gamma}$ for some
  $T\in\mathscr{B}(E^n,E)$. Then $P_j = (P_jT)\Psi_{\Gamma}$, so that
  $t_j\le \|P_jT\|\le \|T\|$ by the monotonicity of the Schauder basis
  for~$E$.

  The first part of the final clause (that~$\Delta_\gamma$ is compact
  if $(t_j)_{j\in\N}$ is unbounded) is immediate from
  Lemma~\ref{fdposdiagopnorm}. We shall prove the second part by
  contraposition.  Suppose that
  $\Delta_\gamma\in\mathscr{L}_{\Gamma}$, so that $\Delta_\gamma =
  S\Psi_{\Gamma}$ for some $S\in\mathscr{B}(E^n,E)$.  Then, for each
  $k\in\N$, we have a commutative diagram
  \[ \spreaddiagramrows{3ex}\spreaddiagramcolumns{6ex}%
  \xymatrix{%
    E\ar^-{\displaystyle{P_k}}[r]%
    \ar^-{\displaystyle{\Delta_\gamma}}[rd]%
    \ar_-{\displaystyle{\Psi_{\Gamma}}}[d]
    &E\\ E^n\ar^-{\displaystyle{S}}[r]
    &E\smashw{,}\ar_-{\displaystyle{\Theta_{\eta(k)}}}[u]} \] where
  $\eta(k) = (t_j^{1/2})_{j=1}^k$ and the operator $\Theta_{\eta(k)}$
  is given by~\eqref{posdiagopinfdimEq1}.  Hence, by the
  definition~\eqref{TKTKNJLthm11June2012Eq1} of~$t_k$ and
  Corollary~\ref{posdiagopinfdim}, we obtain
  \[ t_k\le \|\Theta_{\eta(k)} S\|\le 2t_k^{1/2}\|S\|, \] which
  implies that the sequence $(t_j)_{j\in\N}$ is bounded by $4\|S\|^2$.
\end{proof}

\begin{theorem}\label{Thmlp}
  Let $E$ be a Banach space which is complemented in its bi\-dual and
  has a~Schauder basis. Then $\mathscr{B}(E)$ is the only
  finitely-generated left ideal of~$\mathscr{B}(E)$ which
  contains~$\mathscr{K}(E)$, and hence each finitely-generated,
  maximal left ideal of~$\mathscr{B}(E)$ is fixed.
\end{theorem}
\begin{proof}
  Let $\mathbf{e} = (e_j)_{j\in\N}$ be a Schauder basis for~$E$. By
  passing to an equivalent norm on~$E$, we may suppose
  that~$\mathbf{e}$ is monotone.  Suppose that $\Gamma$ is a
  non-empty, finite subset of~$\mathscr{B}(E)$ such that
  $\mathscr{K}(E)\subseteq\mathscr{L}_{\Gamma}$, and set $n =
  |\Gamma|\in\N$. Lemma~\ref{TKTKNJLthm11June2012} implies that the
  sequence $(t_j)_{j\in\N}$ given by~\eqref{TKTKNJLthm11June2012Eq1}
  is bounded, so that we can find a bounded sequence $(T_j)_{j\in\N}$
  in~$\mathscr{B}(E^n,E)$ such that $P_j = T_j\Psi_{\Gamma}$ for each
  $j\in\N$.

  We may identify $\mathscr{B}(E^n,E^{**})$ with the dual space of the
  projective tensor product $E^n\widehat{\otimes}E^*$; the duality
  bracket is given by
  \[ \langle x\otimes\lambda, S\rangle = \langle \lambda,
  Sx\rangle\quad (x\in E^n,\, \lambda\in E^*,\,
  S\in\mathscr{B}(E^n,E^{**})) \] (\emph{e.g.},
  see~\cite[Proposition~A.3.70]{dales}). Hence
  $\mathscr{B}(E^n,E^{**})$ carries a weak$^*$-topol\-o\-gy, with
  respect to which its unit ball is compact, and so the sequence
  $(\kappa_E T_j)_{j\in\N}$ has a weak$^*$-accumulation point,
  say~$T\in\mathscr{B}(E^{n},E^{**})$. Then, for each $j\in\N$,
  $\lambda\in E^*$, and $\epsilon >0$, we can find an integer $k\ge j$
  such that
  \begin{align*}
    \epsilon &\geq \bigl|\langle\Psi_{\Gamma}e_j\otimes\lambda,
    T-\kappa_E T_k\rangle\bigr|\\ &= \bigl|\langle\lambda, (T -
    \kappa_E T_k)\Psi_{\Gamma}e_j\rangle\bigr| = \bigl|\langle\lambda,
    T\Psi_{\Gamma}e_j -\kappa_E e_j\rangle\bigr|.
  \end{align*}
  Since $\epsilon>0$ and $\lambda\in E^*$ were arbitrary, we conclude
  that $T\Psi_{\Gamma}e_j =\kappa_E e_j$, and therefore
  $T\Psi_{\Gamma} =\kappa_E$.  By the assumption, $\kappa_E$ has a
  left inverse, say
  \mbox{$\Lambda\in\mathscr{B}(E^{**},E)$}. Consequently $I_E =
  (\Lambda T)\Psi_{\Gamma}\in\mathscr{L}_{\Gamma}$, and so
  $\mathscr{L}_{\Gamma} = \mathscr{B}(E)$.
\end{proof}

\begin{example}\label{fgmaxleftidealsinlp}
  Theorem~\ref{Thmlp} implies that, for each of the spaces $E =
  \ell_p$ or $E = L_p[0,1]$, where $p\in(1,\infty)$, $\mathscr{B}(E)$
  is the only finitely-generated left ideal of~$\mathscr{B}(E)$ which
  contains~$\mathscr{K}(E)$, and each finitely-generated, maximal left
  ideal of~$\mathscr{B}(E)$ is fixed. This conclusion is also true for
  $p =1$; indeed, $\ell_1$ is a dual space, and therefore complemented
  in its bi\-dual, while $L_1[0,1]$ is complemented in its bi\-dual
  by~\cite[Theorem~6.3.10]{ak}.

  Many other Banach spaces are known to be complemented in their
  bi\-duals.  The following list gives some examples.
  \begin{romanenumerate}
  \item\label{fgmaxleftidealsinlp1} Let $E$ be a Banach space which is
    isomorphic to a complemented subspace of a dual Banach space; that
    is, for some Banach space~$F$, there are operators
    $U\in\mathscr{B}(E,F^*)$ and $V\in\mathscr{B}(F^*,E)$ with 
    $I_E = VU$. Then the diagram
    \[ \spreaddiagramrows{0ex}\spreaddiagramcolumns{8ex}%
    \xymatrix{%
      E\ar^-{\displaystyle{I_E}}[rr]%
      \ar^-{\displaystyle{\kappa_E}}[rd]%
      \ar_-{\displaystyle{U}}[ddd]
      && E\\ & E^{**}\ar^-{\displaystyle{U^{**}}}[d]\\ &
      F^{***}\ar^-{\displaystyle{\kappa_F^*}}[rd]\\ 
      F^*\ar^-{\displaystyle{\kappa_{F^*}}}[ru]%
      \ar^-{\displaystyle{I_{F^*}}}[rr] &&
      F^*\ar_-{\displaystyle{V}}[uuu]} \] is commutative, which
    implies that the operator $\kappa_EV\kappa_F^*U^{**}$ is a
    projection of~$E^{**}$ on\-to~$\kappa_E(E)$, so that $E$ is
    complemented in its bi\-dual.
  \item As a special case of~\romanref{fgmaxleftidealsinlp1}, suppose
    that~$E$ is a non-zero Banach space for which $\mathscr{B}(E)$ is
    isomorphic to a complemented subspace of a dual Banach
    space. Then, as $\mathscr{B}(E)$ contains a complemented subspace
    isomorphic to~$E$, \romanref{fgmaxleftidealsinlp1}~implies that $E$ is
    complemented in its bi\-dual.

    (It is easy to see that $\mathscr{B}(E)$ contains a complemented
    subspace isomorphic to~$E$. Indeed, 
    choose $\lambda\in E^*$ and $y\in E$ with $\langle
    y,\lambda\rangle = 1$, and consider the operators $U_\lambda\colon
    E\to\mathscr{B}(E)$ and $V_y\colon\mathscr{B}(E)\to E$ given by
    $U_\lambda x = x\otimes\lambda$ for $x\in E$ and $V_y(T) = Ty$
    for $T\in\mathscr{B}(E)$. They satisfy $V_yU_\lambda = I_E$, so
    that~$U_\lambda$ is an isomorphism onto its range, and~$U_\lambda
    V_y$ is a projection of~$\mathscr{B}(E)$ onto the range
    of~$U_\lambda$.) 
  \item\label{fgmaxleftidealsinlp2} Let ~$E$ be a Banach lattice which
    does not contain a subspace isomorphic to~$c_0$. Then~$E$ is
    complemented in its bi\-dual by \cite[Theorem~1.c.4]{lt2}.
  \end{romanenumerate}
\end{example}

\begin{remark}
  Theorem~\ref{Thmlp} does not provide any new information for Banach
  spaces of the form $E = C(K)$, where $K$ is a compact Hausdorff
  space, because the assumption that~$C(K)$ is complemented in its
  bi\-dual implies that~$C(K)$ is injective, so that
  Example~\ref{exampleAutocomplemented}\romanref{exampleAutocomplemented3}
  already applies.
\end{remark}

A slight variation of the proof of Theorem~\ref{Thmlp} gives the
following conclusion.

\begin{theorem}\label{TKNJLoct3Thm1}
  Let $E$ be a non-zero Banach space with a Schauder basis. Then
  $\mathscr{B}(E^*)$ is the only finitely-generated left ideal
  of~$\mathscr{B}(E^*)$ which contains~$\overline{\mathscr{F}(E^*)}$,
  and hence each finitely-generated, maximal left ideal
  of~$\mathscr{B}(E^*)$ is fixed.
\end{theorem}

\begin{proof} 
  Suppose that
  $\overline{\mathscr{F}(E^*)}\subseteq\mathscr{L}_{\Gamma}$,
  where~$\Gamma$ is a non-empty, finite subset of~$\mathscr{B}(E^*)$,
  and set $n = |\Gamma|\in\N$. As in the proof of Theorem~\ref{Thmlp},
  we may suppose that~$E$ has a monotone Schauder
  basis~$\mathbf{e}$. Then, arguing as in the proof of
  Lemma~\ref{TKTKNJLthm11June2012}, we see that the sequence
  \[ t_j' = \inf\bigl\{\| T\| : T\in\mathscr{B}((E^*)^n,E^*),\, P_j^*
  = T\Psi_{\Gamma}\bigr\}\in(0,\infty) \quad (j\in\N) \] is increasing
  and bounded. (Indeed, if the sequence $(t_j')_{j\in\N}$ were
  unbounded, then for $\gamma = \bigl((t_j')^{-1/2}\bigr)_{j\in\N}$,
  we would have
  $\Delta_\gamma^*\in\overline{\mathscr{F}(E^*)}\setminus
  \mathscr{L}_{\Gamma}$, contrary to our assumption.)  Consequently
  there exists a bounded sequence $(T_j)_{j\in\N}$
  in~$\mathscr{B}((E^*)^n,E^*)$ such that $P_j^* = T_j\Psi_{\Gamma}$
  for each $j\in\N$. Let $T$ be a weak$^*$-accumulation point of
  $(T_j)_{j\in\N}$, where we have identified
  $\mathscr{B}((E^*)^n,E^*)$ with the dual space of the projective
  tensor product $E\widehat{\otimes}(E^*)^n$ via the duality bracket
  given by
  \begin{equation}\label{dualitybracket2} \langle x\otimes\mu, S\rangle = 
    \langle x,S\mu\rangle\quad (x\in E,\, \mu\in (E^*)^n,\,
    S\in\mathscr{B}((E^*)^n,E^*)). \end{equation} For each $x\in
  E$, $\lambda\in E^*$, and $\epsilon >0$, we can find $j_0\in\N$ such
  that $$\| x - P_jx\|\le\epsilon(2\|\lambda\|+1)^{-1}$$ whenever $j\ge
  j_0$. Choosing $j\ge j_0$ such that $\bigl|\langle
  x\otimes\Psi_{\Gamma}\lambda, T- T_j\rangle\bigr|\le\epsilon/2$ and
  applying~\eqref{dualitybracket2}, we then obtain
  \begin{align*}
    \bigl|\langle x,(T\Psi_\Gamma - I_{E^*})\lambda\rangle\bigr| &\le
    \bigl|\langle x,(T\Psi_\Gamma - P_j^*)\lambda\rangle\bigr| +
    \bigl|\langle x,(I_E-P_j)^*\lambda\rangle\bigr|\\ &\le
    \bigl|\langle x, (T - T_j)\Psi_{\Gamma}\lambda\rangle\bigr| +
    \bigl|\langle x-P_jx,\lambda\rangle\bigr|\le\epsilon.
  \end{align*}
This implies that $T\Psi_{\Gamma} = I_{E^*}$, and therefore
$\mathscr{L}_{\Gamma} = \mathscr{B}(E^*)$.
\end{proof}

\begin{example} Theorem~\ref{TKNJLoct3Thm1} applies in the following
  two cases which have not already been resolved:
  \begin{romanenumerate} 
  \item $E = X\widehat{\otimes} X^*$, where $X$ is a Banach space with
    a shrinking Schauder basis (this ensures that $E$ has a Schauder
    basis).  Then~$E^*$ is isomorphic to~$\mathscr{B}(X^*)$, so that
    the conclusion is that each finitely-generated, maximal left ideal
    of $\mathscr{B}(\mathscr{B}(X^*))$ is fixed. The most important
    case is where~$X$, and hence~$X^*$, is a separable,
    infinite-dimensional Hilbert space; in this
    case~$\mathscr{B}(X^*)$ does not have the approximation
    property~\cite{Sz}, which gives this example a very different
    flavour from Examples~\ref{exampleAutocomplemented}
    and~\ref{fgmaxleftidealsinlp}, above.
  \item $E = \bigl(\bigoplus_{n\in\N} E_n\bigr)_{\ell_1}$, where
    $(E_n)_{n\in\N}$ is a sequence of Banach spaces with Schauder
    bases whose basis constants are uniformly bounded. Then~$E^*$ is
    isomorphic to~$\bigl(\bigoplus_{n\in\N}
    E_n^*\bigr)_{\ell_\infty}$, and so the conclusion is that each
    finitely-generated, maximal left ideal of
    $\mathscr{B}\bigl(\bigl(\bigoplus E_n^*\bigr)_{\ell_\infty}\bigr)$
    is fixed.
  \end{romanenumerate}
\end{example}

The conditions imposed on the Banach space~$E$ in Theorems~\ref{Thmlp}
and~\ref{TKNJLoct3Thm1} are clearly preserved under the formation of
finite direct sums. In contrast, this need not be the case for the
condition of Corollary~\ref{noproperfgleftidealcontainsFE}.  For
instance, $c_0$ and $\ell_\infty$ both satisfy this condition by
Exam\-ple~\ref{exampleAutocomplemented}%
\romanref{exampleAutocomplemented3}%
--\romanref{exampleAutocomplementedc0}, whereas their direct sum
$c_0\oplus\ell_\infty$ does not.  We shall explore this situation in
greater depth in Section~\ref{nonfixedmaxleftideal}. Notably, as a
particular instance of Theorem~\ref{leftidealgenbyL}, we shall see
that the main conclusion of
Corollary~\ref{noproperfgleftidealcontainsFE} fails for $E =
c_0\oplus\ell_\infty$ because $\mathscr{F}(c_0\oplus\ell_\infty)$
\textsl{is} contained in a proper, closed, singly-generated left ideal
of~$\mathscr{B}(c_0\oplus\ell_\infty)$. 

We do not know the answer to Question~\romanref{DalesQ1} for $E =
c_0\oplus\ell_\infty$, but the following result answers this question
positively for another direct sum arising naturally from
Example~\ref{exampleAutocomplemented}, with the ideal~$\mathscr{S}(E)$
of strictly singular operators taking the role that was played
by~$\mathscr{F}(E)$ in Corollary~\ref{noproperfgleftidealcontainsFE}
and~$\mathscr{K}(E)$ in Theorem~\ref{Thmlp}.

\begin{proposition}\label{c0plusHilbert}
  Let $E = c_0(\mathbb{I})\oplus H$, where $\mathbb{I}$ is a non-empty
  set and $H$ is a Hilbert space.  Then $\mathscr{B}(E)$ is the only
  finitely-generated left ideal of $\mathscr{B}(E)$ which
  contains~$\mathscr{S}(E)$, and hence each finitely-generated,
  maximal left ideal of~$\mathscr{B}(E)$ is fixed.
\end{proposition}
\begin{proof} Let $\mathscr{L}$ be a finitely-generated left ideal of
  $\mathscr{B}(E)$ such that~$\mathscr{S}(E)$ is contained
  in~$\mathscr{L}$. We may suppose that $\mathbb{I}$ is infinite and
  $H$ is infinite-dimensional.  Proposition~\ref{cartesianSGnew}
  implies that~$\mathscr{L}$ is generated by a single operator
  \mbox{$T\in\mathscr{B}(E)$}, say, while
  Corollary~\ref{Sept2011keylemma} shows that $T$ is bounded below and
  thus is an upper semi-Fredholm operator.

  We can represent $T$ as a matrix of operators:
  \[ T = \begin{pmatrix} T_{1,1}\colon c_0(\mathbb{I})\to
    c_0(\mathbb{I}) & T_{1,2}\colon H\to c_0(\mathbb{I})\\
    T_{2,1}\colon c_0(\mathbb{I})\to H & T_{2,2}\colon H\to
    H \end{pmatrix}. \] Each operator from~$H$ to $c_0(\mathbb{I})$ is
  strictly singular because no infinite-dimen\-sional subspace
  of~$c_0(\mathbb{I})$ is isomorphic to a Hilbert space. Similarly,
  each operator from~$c_0(\mathbb{I})$ to~$H$ is strictly
  singular. Hence, by \cite[Proposition~2.c.10]{lt1},
  \[ T - \begin{pmatrix} 0 & T_{1,2}\\ T_{2,1} & 0 \end{pmatrix}
  = \begin{pmatrix} T_{1,1} & 0\\ 0& T_{2,2} \end{pmatrix} \] is an
  upper semi-Fredholm operator, which clearly implies that~$T_{1,1}$
  and~$T_{2,2}$ are upper semi-Fredholm operators. Let
  $P_1\in\mathscr{F}(c_0(\mathbb{I}))$ and
  \mbox{$P_2\in\mathscr{F}(H)$} be projections onto the kernels
  of~$T_{1,1}$ and~$T_{2,2}$,
  respectively. Then \[\widetilde{T}_{1,1}\colon\ x\mapsto
  T_{1,1}x,\quad \ker P_1\to T_{1,1}(c_0(\mathbb{I})), \] is an
  isomorphism, so that $T_{1,1}(c_0(\mathbb{I}))$ is isomorphic
  to~$\ker P_1$, which in turn is isomorphic to~$c_0(\mathbb{I})$
  (because~$\ker P_1$ has finite codimension
  in~$c_0(\mathbb{I})$). Consequently, as in
  Example~\ref{exampleAutocomplemented}%
  \romanref{exampleAutocomplementedc0}, $T_{1,1}(c_0(\mathbb{I}))$ is
  complemented in~$c_0(\mathbb{I})$, so that we can extend the inverse
  of~$\widetilde{T}_{1,1}$ to obtain an operator
  $S_1\in\mathscr{B}(c_0(\mathbb{I}))$ which satisfies $S_1T_{1,1} =
  I_{c_0(\mathbb{I})} - P_1$. Similarly, we can find an operator
  $S_2\in\mathscr{B}(H)$ such that $S_2T_{2,2} = I_H - P_2$. In
  conclusion, we have
  \[ I_E = 
  \begin{pmatrix} S_1 & 0\\ 0 & S_2 \end{pmatrix} T +
  \begin{pmatrix} P_1 & -S_1T_{1,2}\\ -S_2T_{2,1} &
    P_2 \end{pmatrix}\in\mathscr{L} + \mathscr{S}(E) = \mathscr{L}, \]
  and thus $\mathscr{L} = \mathscr{B}(E)$.  

  Since $\mathscr{S}(E)\subseteq\mathscr{E}(E)$, the final clause
  follows immediately from Corollary~\ref{June12dich1}.
\end{proof}

\section{`Exotic' Banach spaces for which each finitely-generated,
  maximal left ideal is fixed}\label{newSection6}
\noindent 
In this section, we shall answer Question~\romanref{DalesQ1}
positively for two classes of custom-made Banach spaces of a
distinctly non-classical nature, using an approach which is completely
different from the one taken in
Section~\ref{sectionAllfgidealsfixed}. More precisely, for each Banach
space~$E$ in either of these two classes, we are able to describe all
the maximal left ideals of~$\mathscr{B}(E)$ explicitly, and it will
then follow easily that only the fixed maximal left ideals are
finitely generated. The reason that we can describe all the maximal
left ideals of~$\mathscr{B}(E)$ is, roughly speaking,
that~$\mathscr{B}(E)$ is `small'. As we shall see, in both cases each
non-fixed, maximal left ideal of~$\mathscr{B}(E)$ is a two-sided ideal
of codimension~one.

We begin with a lemma which can be viewed as a counter-part of
Corollary~\ref{propWCopfg} for left ideals of strictly singular
operators.

\begin{lemma}\label{lemma04042012}
  Let~$E$ be a Banach space, and let~$\mathscr{L}$ be a left ideal
  of~$\mathscr{B}(E)$ such that
  $\mathscr{F}(E)\subseteq\mathscr{L}\subseteq\mathscr{S}(E)$. Then
  the following three conditions are equivalent:
  \begin{alphenumerate}
  \item\label{lemma04042012i} $\mathscr{L}$ is finitely generated;
  \item\label{lemma04042012ii} $\mathscr{L}=\mathscr{B}(E);$
  \item\label{lemma04042012iii} $E$ is finite-dimensional.
\end{alphenumerate}
\end{lemma}
\begin{proof} 
The implications
\alphref{lemma04042012iii}$\Rightarrow$\alphref{lemma04042012ii}%
$\Rightarrow$\alphref{lemma04042012i} are clear.
 
To see that~\alphref{lemma04042012i}
implies~\alphref{lemma04042012iii}, suppose that $\mathscr{L} =
\mathscr{L}_{\Gamma}$ for some non-empty, finite subset $\Gamma$
of~$\mathscr{B}(E)$. Corollary~\ref{Sept2011keylemma} implies that the
operator~$\Psi_{\Gamma}$ is bounded below, while~\eqref{defnOpPsi} and
the fact that $\Gamma\subseteq\mathscr{S}(E)$ show
that~$\Psi_{\Gamma}$ is strictly singular. Hence the do\-main~$E$
of~$\Psi_{\Gamma}$ is finite-dimensional.
\end{proof} 

A Banach space $E$ has \emph{few operators} if $E$ is
infinite-dimensional and each operator on~$E$ is the sum of a scalar
multiple of the identity operator and a strictly singular operator;
that is, $\mathscr{B}(E) = \mathbb{C}I_E + \mathscr{S}(E)$. Gowers and
Maurey~\cite{gm} showed that each hereditarily in\-de\-composable
Banach space has few operators, and constructed the first example of
such a space.

\begin{theorem}\label{fewopsimpliesfixed}
  Let $E$ be a Banach space which has few operators.  Then
  $\mathscr{S}(E)$ is the unique non-fixed, maximal left ideal
  of~$\mathscr{B}(E)$, and $\mathscr{S}(E)$ is not finitely generated
  as a left ideal.
\end{theorem}

\begin{proof} Let $\mathscr{L}$ be a maximal left ideal
  of~$\mathscr{B}(E)$, and suppose that $\mathscr{L}$ is not fixed.
  Then, by Corollary~\ref{June12dich1}, $\mathscr{L}$
  contains~$\mathscr{E}(E)$ and hence~$\mathscr{S}(E)$, which has
  codimension one in~$\mathscr{B}(E)$, so that $\mathscr{L} =
  \mathscr{E}(E) = \mathscr{S}(E)$. This proves the first clause. The
  second clause follows from Lemma~\ref{lemma04042012}.
\end{proof}

To set the scene for our second result, we begin with a short
excursion into the theory of semi-direct products of Banach algebras.
Let $\mathscr{B}$ be a Banach algebra, and let~$\mathscr{C}$
and~$\mathscr{I}$ be a closed subalgebra and a closed, two-sided ideal
of~$\mathscr{B}$, respectively. Then $\mathscr{B}$ is the
\emph{semi-direct product} of $\mathscr{C}$ and~$\mathscr{I}$ if
$\mathscr{C}$ and~$\mathscr{I}$ are complementary subspaces of
$\mathscr{B}$; that is, $\mathscr{C} + \mathscr{I} = \mathscr{B}$ and
$\mathscr{C}\cap\mathscr{I} = \{0\}$. In this case, we denote by
$\rho\colon\mathscr{B}\to\mathscr{C}$ the projection of~$\mathscr{B}$
onto~$\mathscr{C}$ along~$\mathscr{I}$. This is an algebra
homomorphism, as is easy to check.  It is relevant for our purposes
because it induces an isomorphism between the following two lattices
of closed left ideals
\begin{equation}\label{ideallatticeBI}
  \operatorname{Lat}_{\mathscr{I}}(\mathscr{B}) = \bigl\{ \mathscr{L}
  : \mathscr{L}\ \text{is a closed left ideal of}\ \mathscr{B}\
  \text{such that}\
  \mathscr{I}\subseteq\mathscr{L}\bigr\}
\end{equation}
and
\begin{equation}\label{ideallatticeC}
 \operatorname{Lat}(\mathscr{C}) = \bigl\{
  \mathscr{N} : \mathscr{N}\ \text{is a closed left ideal of}\
  \mathscr{C}\bigr\}.
\end{equation}
More precisely, for each
$\mathscr{L}\in\operatorname{Lat}_{\mathscr{I}}(\mathscr{B})$, we have
$\rho(\mathscr{L}) =
\mathscr{L}\cap\mathscr{C}\in\operatorname{Lat}(\mathscr{C})$, and the
mapping \mbox{$\mathscr{L}\mapsto\rho(\mathscr{L})$} is a lattice
isomorphism between  $\operatorname{Lat}_{\mathscr{I}}(\mathscr{B})$
and~$\operatorname{Lat}(\mathscr{C})$; its inverse is given by
\mbox{$\mathscr{N}\mapsto \mathscr{N}+\mathscr{I}$}.  Suppose that the
left ideal
$\mathscr{L}\in\operatorname{Lat}_{\mathscr{I}}(\mathscr{B})$ is
generated by a subset~$\Gamma$ of~$\mathscr{B}$. Then evidently
$\rho(\mathscr{L})$ is generated by the subset~$\rho(\Gamma)$
of~$\mathscr{C}$, so that~$\rho$ maps each closed, finitely-generated
left ideal of~$\mathscr{B}$ containing~$\mathscr{I}$ to a closed,
finitely-generated left ideal of~$\mathscr{C}$.
 
We shall next state two classical results about $C(K)$-spaces. The
first is due to Pe{\l}czy{\a'n}\-ski \cite[Theorem~1]{pel}, and
characterizes the weakly compact operators from a $C(K)$-space into an
arbitrary Banach space.
\begin{theorem}\label{PelcWCK}
  Let~$K$ be a non-empty, compact Hausdorff space, and let $E$ be a Banach
  space. Then the following three conditions are equivalent for each
  operator $T\in\mathscr{B}(C(K),E)\colon$
  \begin{alphenumerate}
  \item $T$ is weakly compact;
  \item $T$ is strictly singular;
  \item $T$ does not fix a copy of~$c_0$.
  \end{alphenumerate}
\end{theorem}

The second result describes the maximal ideals of the Banach
algebra~$C(K)$, as well as the finitely-generated ones. (Note that the
notions of a left, right, and two-sided ideal coincide in $C(K)$
because $C(K)$ is commutative.)  Given a point $k\in K$, we write
$\epsilon_k\colon C(K)\to\C$ for the evaluation map at~$k$; that is,
$\epsilon_k(f)= f(k)$ for each $f\in C(K)$. This is a surjective
algebra homomorphism of norm one.

\begin{theorem}\label{maxidealsofCK}
  Let~$K$ be a compact Hausdorff space. Then:
  \begin{romanenumerate}
  \item\label{maxidealsofCK1} each maximal ideal of~$C(K)$ has the
    form $\ker\epsilon_k$ for a unique point $k\in K;$
  \item\label{maxidealsofCK2} the maximal ideal $\ker\epsilon_k$ is
    finitely generated if and only if the point~$k$ is isolated
    in~$K$.
  \end{romanenumerate}
\end{theorem}
\begin{proof}
  The first clause is folklore (\emph{e.g.}, see
  \cite[Theorem~4.2.1(i)]{dales}), while the second is the
  complex-valued counter-part of a classical theorem of Gillman
  \cite[Corollary~5.4]{gillman}. Both clauses are also easy to verify
  directly.
\end{proof}

We require one further notion before we can present our result.  For a
non-empty, compact Hausdorff space~$K$ and a function $g\in C(K)$, we denote by
$M_g\in\mathscr{B}(C(K))$ the multiplication operator given by~$g$;
that is, $M_gf = gf$ for each \mbox{$f\in C(K)$}. The mapping
\begin{equation}\label{defnMU}
  \mu\colon\ g\mapsto M_g,\quad C(K)\to\mathscr{B}(C(K)), 
\end{equation}
is an isometric, unital algebra homomorphism. An operator
$T\in\mathscr{B}(C(K))$ is a \emph{weak multi\-pli\-ca\-tion} if it
has the form \mbox{$T = M_g + S$} for some $g\in C(K)$ and
$S\in\mathscr{W}(C(K))$. The fourth-named
author \cite[Theorem~6.1]{koszmider} (assuming the continuum
hypothesis) and Ple\-ba\-nek \cite[Theorem~1.3]{plebanek} (without any
assumptions beyond ZFC) have constructed an~example of a connected,
compact Hausdorff space~$K$ for which each operator on~$C(K)$ is a
weak multiplication. This ensures that the following theorem is not
vacuous.

\begin{theorem}\label{CKweakmult}
  Let~$K$ be a compact Hausdorff space without isolated points and
  such that each operator on~$C(K)$ is a weak multiplication. 
  \begin{romanenumerate}
  \item\label{CKweakmult1} The Banach algebra $\mathscr{B}(C(K))$ is
    the semi-direct product of the sub\-algebra $\mu(C(K))$ and the
    ideal~$\mathscr{W}(C(K))$, where~$\mu$ is the homomorphism given
    by~\eqref{defnMU}. 
  \item\label{CKweakmult2} Let $\mathscr{L}$ be a subset
    of~$\mathscr{B}(C(K))$.  Then the following four conditions are
    equivalent:
    \begin{alphenumerate}
    \item\label{CKweakmult2b} $\mathscr{L}$ is a non-fixed, maximal
      left ideal of~$\mathscr{B}(C(K));$
    \item\label{CKweakmult2a} $\mathscr{L}$ is a maximal left ideal
      of~$\mathscr{B}(C(K))$, and $\mathscr{L}$ is not finitely
      generated;
    \item\label{CKweakmult2d} $\mathscr{L}$ is a maximal two-sided
      ideal of~$\mathscr{B}(C(K));$
    \item\label{CKweakmult2c} $\mathscr{L} = \bigl\{ M_g + S :
      S\in\mathscr{W}(C(K))\ {\normalfont{\text{and}}}\ g\in C(K)\
      {\normalfont{\text{with}}}\ g(k) = 0\bigr\}$ for some $k\in K$.
    \end{alphenumerate}
    In the positive case, the point $k\in K$ such
    that~\alphref{CKweakmult2c} holds is uniquely determined
    by~$\mathscr{L}$.
  \end{romanenumerate}
\end{theorem}
\begin{proof} 
  \romanref{CKweakmult1}. We have $\mathscr{B}(C(K)) = \mu(C(K)) +
  \mathscr{W}(C(K))$ because each operator on~$C(K)$ is a weak
  multi\-pli\-ca\-tion.  Theorem~\ref{PelcWCK} allows us to
  replace~$\mathscr{W}(C(K))$ with~$\mathscr{S}(C(K))$, which we shall
  do in the remainder of this proof because the latter ideal suits our
  approach better.

  To see that $\mu(C(K))\cap \mathscr{S}(C(K)) = \{0\}$, suppose that
  \mbox{$g\in C(K)\setminus\{0\}$}. Take $k_0\in K$ such that
  $g(k_0)\ne0$, set $\epsilon = |g(k_0)|/2>0$, and choose an open
  neighbourhood~$N$ of~$k_0$ such that $|g(k)|\ge\epsilon$ for each
  $k\in N$.  Using Urysohn's lemma and the fact that $k_0$ is not
  isolated in~$K$, we deduce that the subspace \[ F = \bigl\{ f\in
  C(K) : f(k) = 0\quad (k\in K\setminus N)\bigr\} \] of~$C(K)$ is
  infinite-dimensional. Since
  \[ \| M_g f\| = \sup\bigl\{ |g(k)f(k)| : k\in N\bigr\}\ge \epsilon
  \| f\|\quad (f\in F), \] we conclude that $M_g$ is not strictly
  singular, as required.

  \romanref{CKweakmult2}.  For each $k\in K$, let \[ \mathscr{Z}_k =
  \mu(\ker\epsilon_k) + \mathscr{S}(C(K)),\] so that
  \[ \mathscr{Z}_k = \bigl\{ M_g + S : S\in\mathscr{S}(C(K))\
  \text{and}\ g\in C(K)\ \text{with}\ g(k) = 0\bigr\}. \]
  By~\romanref{CKweakmult1}, $\mathscr{Z}_k$ is a two-sided ideal of
  codimension one in~$\mathscr{B}(C(K))$, and thus is maximal both as
  a left and a two-sided ideal. The implication
  \alphref{CKweakmult2c}$\Rightarrow$\alphref{CKweakmult2d} is now
  immediate, while
  \alphref{CKweakmult2c}$\Rightarrow$\alphref{CKweakmult2a} follows
  because $\rho(\mathscr{Z}_k) = \mu(\ker\epsilon_k)$ is not finitely
  generated by Theorem~\ref{maxidealsofCK}\romanref{maxidealsofCK2},
  so that $\mathscr{Z}_k$ is not finitely generated as a left ideal,
  as explained in the paragraph following~\eqref{ideallatticeC}.

  The implication
  \alphref{CKweakmult2a}$\Rightarrow$\alphref{CKweakmult2b} is clear
  because each fixed, maximal left ideal is finitely generated by
  Proposition~\ref{trivialfgmaxrightideals}%
  \romanref{trivialfgmaxrightideals1}.

  \alphref{CKweakmult2b}$\Rightarrow$\alphref{CKweakmult2c}.  Suppose
  that~$\mathscr{L}$ is a non-fixed, maximal left ideal
  of~$\mathscr{B}(C(K))$. Then, by Corollary~\ref{June12dich1},
  $\mathscr{L}$~con\-tains~$\mathscr{E}(C(K))$ and
  thus~$\mathscr{S}(C(K))$, so that $\mathscr{L}$ is a maximal element
  of the
  lattice~$\operatorname{Lat}_{\mathscr{S}(C(K))}(\mathscr{B}(C(K)))$
  given by~\eqref{ideallatticeBI}. Hence, in the notation
  of~\eqref{ideallatticeC}, there is a maximal element~$\mathscr{N}$
  of the lattice~$\operatorname{Lat}(\mu(C(K)))$ such that
  $\mathscr{L} = \mathscr{N} + \mathscr{S}(C(K))$.
  Theorem~\ref{maxidealsofCK}\romanref{maxidealsofCK1} implies that
  $\mathscr{N} = \mu(\ker\epsilon_k)$ for some $k\in K$, and
  consequently $\mathscr{L} = \mathscr{Z}_k$.

  \alphref{CKweakmult2d}$\Rightarrow$\alphref{CKweakmult2c}. Suppose
  that~$\mathscr{L}$ is a maximal two-sided ideal
  of~$\mathscr{B}(C(K))$. Then, as mentioned in
  Remark~\ref{inessin2sidedmaxideals}, $\mathscr{L}$
  contains~$\mathscr{E}(C(K))$ and hence~$\mathscr{L}$ contains
  $\mathscr{S}(C(K))$, so that $\mathscr{L} = \mathscr{Z}_k$ for some
  $k\in K$ by~\romanref{CKweakmult1} and
  Theorem~\ref{maxidealsofCK}\romanref{maxidealsofCK1}.

  The final clause follows because
  $\ker\epsilon_{k_1}\neq\ker\epsilon_{k_2}$ whenever $k_1,k_2\in K$
  are distinct, and hence also
  $\mathscr{Z}_{k_1}\neq\mathscr{Z}_{k_2}$.
\end{proof}

\begin{remark}\label{statusforlp}
  Example~\ref{exampleAutocomplemented}%
  \romanref{exampleAutocomplemented3}--%
  \romanref{exampleAutocomplementedc0} and Theorem~\ref{CKweakmult}
  show that there are compact Hausdorff spaces~$K$ such that
  Question~\romanref{DalesQ1} has a positive answer for $E =
  C(K)$. However, this question remains open for some very important
  $C(K)$-spaces. Indeed, it is known that~$C(K)$ contains a closed
  subspace which is isomorphic to~$C(K)$ and which is not complemented
  in~$C(K)$ for each of the following compact Hausdorff
  spaces~$K\colon$
  \begin{romanenumerate}
  \item\label{statusforlp2} $K = [0,1]$ (see~\cite{amir});
  \item\label{statusforlp2.5} $K = [0,\alpha]$ for any ordinal
    $\alpha\ge\omega^\omega$, where $[0,\alpha]$ denotes the set of
    ordinals less than or equal to~$\alpha$, equipped with the order
    topology. (Baker~\cite{baker} showed this in the case where
    $\alpha = \omega^\omega$; the conclusion for
    general~$\alpha\ge\omega^\omega$ follows immediately from Baker's
    result because $C[0,\alpha]$ is isomorphic to~$C[0,\alpha]\oplus
    C[0,\omega^\omega]$.)
  \end{romanenumerate}
  Hence, by Corollary~\ref{Sept2011keylemma}, $\mathscr{F}(C(K))$ is
  contained in a singly-generated, proper left ideal
  of~$\mathscr{B}(C(K))$ for each of these~$K$, but we do not know
  whether such a left ideal can be chosen also to be maximal (or even
  closed). 

  This question cannot be answered by a variant of Theorem~\ref{Thmlp}
  because we can strengthen the above conclusion to obtain that
  $\mathscr{K}(C(K))$ is contained in a singly-generated, proper left
  ideal of~$\mathscr{B}(C(K))$ for each of the above~$K$. To see this,
  take an operator $U\in\mathscr{B}(C(K))$ which is bounded below and
  whose range $F = U(C(K))$ is not complemented in~$C(K)$, and
  consider the isomorphism $\widetilde{U}\colon x\mapsto Ux,\, C(K)\to
  F$. Then, for each \mbox{$S\in\mathscr{K}(C(K))$}, the operator
  \mbox{$S\widetilde{U}^{-1}\colon F\to C(K)$} has an extension
  $T\in\mathscr{K}(C(K))$ by a theorem of Grothen\-dieck
  (see~\cite[pp.~559--560]{groth}, or \cite[Theorem~1]{linden}). Hence
  we have $S = TU$, and con\-sequent\-ly
  \mbox{$\mathscr{K}(C(K))\subseteq\mathscr{L}_{\{ U\}}$}.
\end{remark}

\section{A non-fixed and singly-generated, maximal left ideal of
  operators}\label{nonfixedmaxleftideal}
\noindent
The main aim of this section is to prove
Theorems~\ref{Thmnonfixedmaxleftideal}
and~\ref{Thmnonfixedmaxleftideal2}. Several parts of those theorems
are special cases of more general results, which may be of independent
interest, and so we shall take a more general approach, specializing
only when we need to.

Recall that, for a non-empty set~$\mathbb{I}$, we denote by
$\ell_\infty(\mathbb{I})$ the Banach space of bounded, complex-valued
functions defined on~$\mathbb{I}$, and $\ell_\infty =
\ell_\infty(\N)$. Our first result collects some known facts about
operators from~$\ell_\infty(\mathbb{I})$ that we shall use several times.
\begin{lemma}\label{operatorsXtolinftywc}
  Let $\mathbb{I}$ be a non-empty set, and let $X$ be a Banach space.
  \begin{romanenumerate}
  \item\label{operatorsXtolinftywc1} An operator from
    $\ell_\infty(\mathbb{I})$ to~$X$ is weakly compact if and only if it
    is strictly singular.
  \item\label{operatorsXtolinftywc2} Suppose that the set $\mathbb{I}$ is
    infinite. Then each opera\-tor from $\ell_\infty(\mathbb{I})$ to~$X$
    is weakly compact if and only if $X$ does not contain a subspace
    iso\-mor\-phic to~$\ell_\infty$.
  \end{romanenumerate}
\end{lemma}
\begin{proof}
  \romanref{operatorsXtolinftywc1}. This is a special case of
  Theorem~\ref{PelcWCK}.

  \romanref{operatorsXtolinftywc2}.  The hard part is the
  implication~$\Leftarrow$, which however follows immediately
  from~\cite[Proposition~2.f.4]{lt1}.

  The forward implication is straightforward. Indeed, suppose
  contra\-pos\-i\-tive\-ly that $X$ contains a sub\-space which is
  iso\-mor\-phic to~$\ell_\infty$, and take an operator
  $U\in\mathscr{B}(\ell_\infty,X)$ which is bounded below. Choose an
  injective mapping $\theta\colon\N\to\mathbb{I}$, and define an
  operator
  $C_\theta\in\mathscr{B}(\ell_\infty(\mathbb{I}),\ell_\infty)$ by
  $C_\theta f = f\circ\theta$ for each
  $f\in\ell_\infty(\mathbb{I})$. Then $UC_\theta$ is not weakly
  compact, for instance because it fixes a copy of~$\ell_\infty$.
\end{proof}

In the remainder of this section we shall consider a Banach space $X$
such that
\begin{capromanenumerate}
\item\label{assumption1} the bi\-dual of~$X$ is isomorphic
  to~$\ell_\infty(\mathbb{I}_1)$ for some infinite set~$\mathbb{I}_1$ via a
  fixed isomorphism $V\colon X^{**}\to\ell_\infty(\mathbb{I}_1)$; and
\item\label{assumption2} no subspace of~$X$ is isomorphic
  to~$\ell_\infty$.
\end{capromanenumerate}
For example, $X = c_0$ satisfies both of these conditions with
$\mathbb{I}_1 = \N$.

Let $\mathbb{I}_2$ be a disjoint copy of~$\mathbb{I}_1$ (that is,
$\mathbb{I}_2$ is a set of the same cardinality as~$\mathbb{I}_1$ and
satisfies $\mathbb{I}_1\cap \mathbb{I}_2 = \emptyset$), and set $\mathbb{I}
= \mathbb{I}_1\cup\mathbb{I}_2$.  We consider $\ell_\infty(\mathbb{I}_1)$
and $\ell_\infty(\mathbb{I}_2)$ as complementary subspaces of
$\ell_\infty(\mathbb{I})$ in the natural way, and denote by $P_1$ and
$P_2$ the corresponding projections of $\ell_\infty(\mathbb{I})$ onto
$\ell_\infty(\mathbb{I}_1)$ and $\ell_\infty(\mathbb{I}_2)$, respectively.
Moreover, we shall choose a~bijec\-tion
$\phi\colon\mathbb{I}_2\to\mathbb{I}$; we then obtain an isometric
isomorphism $C_\phi$ of~$\ell_\infty(\mathbb{I})$ onto the
sub\-space~$\ell_\infty(\mathbb{I}_2)$ by the definition $C_\phi f =
f\circ\phi$ for each $f\in\ell_\infty(\mathbb{I})$.

Let $E = X\oplus\ell_\infty(\mathbb{I})$ with norm
$\bigl\|(x,f)\bigr\|_E = \max\bigl\{\|x\|_X,\|f\|_\infty\bigr\}$.  We
identify operators~$T$ on~$E$ with $(2\times 2)$-matrices
\[ \begin{pmatrix} T_{1,1}\colon X\to X & T_{1,2}\colon
  \ell_\infty(\mathbb{I})\to X\\ T_{2,1}\colon X\to
  \ell_\infty(\mathbb{I}) & T_{2,2}\colon
  \ell_\infty(\mathbb{I})\to\ell_\infty(\mathbb{I}) \end{pmatrix}. \] Note
that assumption~\romanref{assumption2} and
Lemma~\ref{operatorsXtolinftywc}\romanref{operatorsXtolinftywc2} imply
that the operator $T_{1,2}$ is always weakly compact. This fact will
play a key role for us.

Despite our focus on left ideals, our first result about the Banach
space~$E$ is concerned with two-sided ideals.
\begin{proposition}\label{twosidedidealsinBE}
  \begin{romanenumerate}
  \item\label{defnidealI1} The set
    \begin{equation*}
      \mathscr{W}_1 = \biggl\{\begin{pmatrix} T_{1,1} & T_{1,2}\\ 
        T_{2,1} & T_{2,2} \end{pmatrix}\in\mathscr{B}(E) :
      T_{1,1}\in\mathscr{W}(X)\biggr\} \end{equation*}
    is a proper, closed two-sided ideal of~$\mathscr{B}(E)$, and
    $\mathscr{W}_1$ is a maximal two-sided ideal of~$\mathscr{B}(E)$
    if and only if $\mathscr{W}(X)$ is a maximal  two-sided ideal
    of~$\mathscr{B}(X)$. 
  \item\label{defnidealI2} The set \begin{equation*} \mathscr{W}_2 =
      \biggl\{\begin{pmatrix} T_{1,1} & T_{1,2}\\ T_{2,1} &
        T_{2,2} \end{pmatrix}\in\mathscr{B}(E) :
      T_{2,2}\in\mathscr{W}(\ell_\infty(\mathbb{I}))\biggr\} \end{equation*}
    is a proper, closed two-sided ideal of~$\mathscr{B}(E)$, and the
    following three conditions are equivalent:
    \begin{alphenumerate}
    \item\label{defnidealI2a} $\mathscr{W}_2$ is a maximal  two-sided
      ideal of~$\mathscr{B}(E);$ 
    \item\label{defnidealI2b} $\mathscr{W}(\ell_\infty(\mathbb{I}))$ is a
      maximal two-sided ideal of~$\mathscr{B}(\ell_\infty(\mathbb{I}));$
    \item\label{defnidealI2c} $\mathbb{I}$ is countable.
    \end{alphenumerate}
  \end{romanenumerate}
\end{proposition}
\begin{proof} 
  \romanref{defnidealI1}.  The mapping
  \[ T\mapsto T_{1,1} + \mathscr{W}(X),\quad
  \mathscr{B}(E)\to\mathscr{B}(X)\big/\mathscr{W}(X), \] is a
  surjective algebra homomorphism of norm one. Hence its kernel, which
  is equal to~$\mathscr{W}_1$, is a closed two-sided ideal
  of~$\mathscr{B}(E)$. This ideal is proper because $X$~is
  non-reflexive by assumption~\romanref{assumption1}. The fundamental
  isomorphism theorem implies that the Banach algebras
  $\mathscr{B}(E)\big/\mathscr{W}_1$
  and~$\mathscr{B}(X)\big/\mathscr{W}(X)$ are isomorphic, and so
  $\mathscr{B}(E)\big/\mathscr{W}_1$ is simple if and only if
  $\mathscr{B}(X)\big/\mathscr{W}(X)$ is simple. Consequently
  $\mathscr{W}_1$ is a maximal two-sided ideal of $\mathscr{B}(E)$ if
  and only if $\mathscr{W}(X)$ is a maximal two-sided ideal
  of~$\mathscr{B}(X)$.

  \romanref{defnidealI2}.  An obvious modification of the argument
  given above shows that $\mathscr{W}_2$ is a proper, closed two-sided
  ideal of~$\mathscr{B}(E)$, and that
  conditions~\alphref{defnidealI2a} and~\alphref{defnidealI2b} are
  equivalent.  The im\-pli\-ca\-tion
  \alphref{defnidealI2c}$\Rightarrow$%
  \alphref{defnidealI2b} follows from~\cite[Proposition~2.f.4]{lt1}.

  Conversely, to prove that \alphref{defnidealI2b}$\Rightarrow$%
  \alphref{defnidealI2c}, suppose that
  $\mathscr{W}(\ell_\infty(\mathbb{I}))$ is a maximal two-sided ideal
  of~$\mathscr{B}(\ell_\infty(\mathbb{I}))$, and denote by
  $\mathscr{G}_{\ell_\infty}(\ell_\infty(\mathbb{I}))$ the set of
  operators on~$\ell_\infty(\mathbb{I})$ that factor
  through~$\ell_\infty$. This is a two-sided ideal
  of~$\mathscr{B}(\ell_\infty(\mathbb{I}))$ because $\ell_\infty$ is
  isomorphic to \mbox{$\ell_\infty\oplus\ell_\infty$}. Hence
  $\mathscr{G}_{\ell_\infty}(\ell_\infty(\mathbb{I})) +
  \mathscr{W}(\ell_\infty(\mathbb{I}))$ is also a two-sided ideal,
  which is strictly greater than
  $\mathscr{W}(\ell_\infty(\mathbb{I}))$ because
  $\ell_\infty(\mathbb{I})$ contains a complemented copy
  of~$\ell_\infty$, and any projection with range isomorphic
  to~$\ell_\infty$ belongs to
  $\mathscr{G}_{\ell_\infty}(\ell_\infty(\mathbb{I}))\setminus
  \mathscr{W}(\ell_\infty(\mathbb{I}))$. Con\-sequent\-ly, by the
  maximality of~$\mathscr{W}(\ell_\infty(\mathbb{I}))$, there are
  operators $R\in\mathscr{G}_{\ell_\infty}(\ell_\infty(\mathbb{I}))$
  and $S\in \mathscr{W}(\ell_\infty(\mathbb{I}))$ such that
  $I_{\ell_\infty(\mathbb{I})} = R+S$.  Then $R =
  I_{\ell_\infty(\mathbb{I})}-S$ is a Fredholm operator by
  \cite[Proposition~2.c.10]{lt1} and Lemma~\ref{operatorsXtolinftywc}%
  \romanref{operatorsXtolinftywc1}, and this implies that
  $I_{\ell_\infty(\mathbb{I})} = URT$ for some
  operators~$T,U\in\mathscr{B}(\ell_\infty(\mathbb{I}))$
  because~$\ell_\infty(\mathbb{I})$ is isomorphic to its
  hyper\-planes. Thus the identity operator on~$\ell_\infty(\mathbb{I})$
  factors through~$\ell_\infty$, which is possible only
  if~$\mathbb{I}$ is countable.
\end{proof}

Set \begin{equation}\label{defnL} L = \begin{pmatrix} 0 & 0\\
    V\kappa_X & C_\phi \end{pmatrix}\in\mathscr{B}(E), \end{equation}
where the operators~$V$ and $C_\phi$ were introduced on
p.~\pageref{assumption1}.  Since the ranges of $V$ and~$C_\phi$ are
contained in the complementary subspaces~$\ell_\infty(\mathbb{I}_1)$
and~$\ell_\infty(\mathbb{I}_2)$ of~$\ell_\infty(\mathbb{I})$,
respectively, we have
\begin{align*}
  \|L(x,f)\|_E &= \|V\kappa_X x + C_\phi f\|_\infty  =
  \max\bigl\{\|V\kappa_X x\|_\infty,\| C_\phi
  f\|_\infty\bigr\}\\ &= \max\bigl\{\|V\kappa_X
  x\|_\infty,\|f\|_\infty\bigr\}\quad
  (x\in X,\,f\in\ell_\infty(\mathbb{I})),
\end{align*}
which shows that the operator $L$ is bounded below because $V\kappa_X$
is bounded below.  This conclusion is also immediate from our next
result and Corollary~\ref{Sept2011keylemma}.

\begin{theorem}\label{leftidealgenbyL}
  The ideal $\mathscr{W}_1$ defined in
  Proposition~{\normalfont{\ref{twosidedidealsinBE}}}%
  \romanref{defnidealI1} is the left ideal generated by the operator
  $L$ given by~\eqref{defnL}{\normalfont{;}} that is,
  \[ \mathscr{W}_1 = \mathscr{L}_{\{L\}}. \]
\end{theorem}
\begin{proof}
  We have $L\in\mathscr{W}_1$ because $L_{1,1} = 0$, and hence the
  inclusion $\supseteq$ follows.

  We shall prove the reverse inclusion in three steps. First, we see
  that
  \[ \begin{pmatrix} 0 & 0\\ 0 &
    I_{\ell_\infty(\mathbb{I})} \end{pmatrix} = \begin{pmatrix} 0 & 0\\
    0 & C_\phi^{-1}P_2 \end{pmatrix} \begin{pmatrix} 0 & 0\\
    V\kappa_X & 
    C_\phi \end{pmatrix}\in \mathscr{L}_{\{L\}}, \] and consequently we have
  \begin{equation}\label{leftidealgenbyLeq1}
    \begin{pmatrix} 0 & T_{1,2}\\ 0 & T_{2,2}\end{pmatrix} = 
    \begin{pmatrix} 0 & T_{1,2}\\ 0 &
      T_{2,2} \end{pmatrix}\begin{pmatrix} 0 & 0\\ 0 & 
      I_{\ell_\infty(\mathbb{I})} \end{pmatrix}\in\mathscr{L}_{\{L\}}
  \end{equation}   
  for each $T_{1,2}\in\mathscr{B}(\ell_\infty(\mathbb{I}),X)$ and
  $T_{2,2}\in\mathscr{B}(\ell_\infty(\mathbb{I}))$.
  
  Second, let $T_{2,1}\in\mathscr{B}(X,\ell_\infty(\mathbb{I}))$.  Being
  bounded below, the operator $V\kappa_X$ is an isomorphism onto its
  range $Y := V\kappa_X(X)\subseteq\ell_\infty(\mathbb{I}_1)$, so that
  it has an inverse $R\in\mathscr{B}(Y,X)$. By the injectivity of
  $\ell_\infty(\mathbb{I}_1)$, the composite operator
  $T_{2,1}R\in\mathscr{B}(Y,\ell_\infty(\mathbb{I}))$ extends to
  an~opera\-tor
  $S\in\mathscr{B}(\ell_\infty(\mathbb{I}_1),\ell_\infty(\mathbb{I}))$,
  which then satisfies $SV\kappa_X = T_{2,1}$. Hence we have
  \begin{equation}\label{leftidealgenbyLeq2}
    \begin{pmatrix} 0 & 0\\ T_{2,1} & 0\end{pmatrix} =  
    \begin{pmatrix} 0 & 0\\ 0 &
      SP_1 \end{pmatrix}\begin{pmatrix} 0 & 0\\ V\kappa_X &
      C_\phi \end{pmatrix}\in\mathscr{L}_{\{L\}}. \end{equation}  

  Third, each operator $T_{1,1}\in\mathscr{W}(X)$ satisfies
  $T_{1,1}^{**}(X^{**})\subseteq\kappa_X(X)$ (\emph{e.g.}, see
  \cite[Theo\-rem~3.5.8]{meg}). We can therefore define an operator
  $U\in\mathscr{B}(\ell_\infty(\mathbb{I}_1),X)$ by \[ Uf =
  \kappa_X^{-1}T_{1,1}^{**}V^{-1}f\quad
  (f\in\ell_\infty(\mathbb{I}_1)). \] Since $\kappa_X U V\kappa_X =
  T_{1,1}^{**}\kappa_X = \kappa_X T_{1,1}$, we have $UV\kappa_X=
  T_{1,1}$, and so
 \begin{equation}\label{leftidealgenbyLeq3} 
 \begin{pmatrix} T_{1,1} & 0\\ 0 & 0\end{pmatrix} =  
 \begin{pmatrix} 0 & UP_1\\ 0 &
   0 \end{pmatrix} \begin{pmatrix} 0 & 0\\ V\kappa_X &
   C_\phi \end{pmatrix}\in\mathscr{L}_{\{L\}}. \end{equation}  

Combining \eqref{leftidealgenbyLeq1}--\eqref{leftidealgenbyLeq3}, we
conclude that each operator $T\in\mathscr{W}_1$ belongs
to~$\mathscr{L}_{\{L\}}$.
\end{proof}

\begin{remark}
  Since the operator $L$ given by~\eqref{defnL} is bounded below and
  generates a proper left ideal of~$\mathscr{B}(E)$, its range is not
  complemented in~$E$ by Lemma~\ref{properidealgeneratedlemma}.  This
  is also easy to verify directly.
\end{remark}

A Banach space $F$ has \emph{very few operators} if $F$ is
infinite-dimensional and each operator on~$F$ is the sum of a scalar
multiple of the identity operator and a compact operator; that is,
$\mathscr{B}(F) = \mathbb{C}I_F + \mathscr{K}(F)$. Argyros and
Haydon~\cite{ah} constructed the first example of a~Banach
space~$X_{\text{\normalfont{AH}}}$ which has very few operators. We
shall now specialize to the case where $X =
X_{\text{\normalfont{AH}}}$. The following result collects the
properties of $X_{\text{\normalfont{AH}}}$ that we shall~require.

\begin{theorem}[Argyros and Haydon]\label{thmAH} There is a
  Banach space~$X_{\text{\normalfont{AH}}}$ with the following three
  properties:
  \begin{romanenumerate}
  \item\label{thmAH1} $X_{\text{\normalfont{AH}}}$ has very few
    operators;
  \item\label{thmAH2} $X_{\text{\normalfont{AH}}}$ has a Schauder
    basis;
  \item\label{thmAH3} the dual space of $X_{\text{\normalfont{AH}}}$
    is isomorphic to~$\ell_1$.
  \end{romanenumerate}
\end{theorem}

Using this, we can easily prove Theorem~\ref{Thmnonfixedmaxleftideal}.

\begin{proof}[Proof of
  Theorem~{\normalfont{\ref{Thmnonfixedmaxleftideal}}}] 
  We begin by checking that $X_{\text{\normalfont{AH}}}$ satisfies the
  two assumptions made on p.~\pageref{assumption1}: indeed, 
  Theorem~\ref{thmAH}\romanref{thmAH3} ensures that
  $X_{\text{\normalfont{AH}}}^{**}$ is isomorphic to~$\ell_\infty$,
  while Theorem~\ref{thmAH}\romanref{thmAH2} (or~\romanref{thmAH3})
  implies that $X_{\text{\normalfont{AH}}}$ does not
  contain~$\ell_\infty$. Moreover, we  see that 
  $\mathscr{W}(X_{\text{\normalfont{AH}}}) =
  \mathscr{K}(X_{\text{\normalfont{AH}}})$ because
  Theorem~\ref{thmAH}\romanref{thmAH3} implies that
  $X_{\text{\normalfont{AH}}}$ is non-reflexive, so that
  $\mathscr{W}(X_{\text{\normalfont{AH}}})$ is a closed, non-zero,
  proper two-sided ideal of~$\mathscr{B}(X_{\text{\normalfont{AH}}})$,
  and $\mathscr{K}(X_{\text{\normalfont{AH}}})$ is the \textsl{only}
  such ideal by
  Theorem~\ref{thmAH}\romanref{thmAH1}--\romanref{thmAH2}. Hence the
  set $\mathscr{K}_1$ given by~\eqref{idealK1} is equal to the
  ideal~$\mathscr{W}_1$ defined in
  Proposition~{\normalfont{\ref{twosidedidealsinBE}}}%
  \romanref{defnidealI1}, and~$\mathscr{W}_1$ is singly generated as a
  left ideal by Theorem~\ref{leftidealgenbyL}.
  Theorem~\ref{thmAH}\romanref{thmAH1} implies that $\mathscr{K}_1$
  has codimension one in~$\mathscr{B}(E)$, so that it is trivially
  maximal as a left, right, and two-sided ideal. (The latter also
  follows from Proposition~{\normalfont{\ref{twosidedidealsinBE}}}%
  \romanref{defnidealI1}.) Being a non-zero, two-sided ideal,
  $\mathscr{K}_1$~contains $\mathscr{F}(E)$, and therefore
  $\mathscr{K}_1$ is not fixed.
\end{proof}

\begin{remark} \begin{romanenumerate}
  \item The Banach space $E =
    X_{\text{\normalfont{AH}}}\oplus\ell_\infty$ is clearly
    non-sep\-a\-ra\-ble, so the question naturally arises whether a
    separable Banach space~$E$ exists such that~$\mathscr{B}(E)$
    contains a non-fixed, finitely generated maximal left ideal. This
    has recently been answered affirmatively~\cite{kanialaustsen}.
  \item Proposition~\ref{c0plusHilbert} and the discussion preceding
    it raise the question whether the class of Banach spaces for which
    Question~\romanref{DalesQ1} has a positive answer is closed under
    finite direct sums.  Theorem~\ref{Thmnonfixedmaxleftideal} implies
    that this is not the case because~$X_{\text{\normalfont{AH}}}$
    and~$\ell_\infty$ both belong to this class by
    Theorem~\ref{fewopsimpliesfixed} and
    Example~\ref{exampleAutocomplemented}%
    \romanref{exampleAutocomplemented3}, respectively, whereas their
    direct sum does not.
  \end{romanenumerate}
\end{remark}

We shall next give a characterization of the ideal~$\mathscr{K}_1$
defined by~\eqref{idealK1}. Theorem~\ref{Thmnonfixedmaxleftideal2}
will be an easy consequence of this result.

\begin{theorem}\label{propCharleftidealcontainingFE}
  Let $E=X_{\text{\normalfont{AH}}}\oplus\ell_\infty$. Then the
  following three conditions are equivalent for each
  subset~$\mathscr{L}$ of~$\mathscr{B}(E)\colon$
  \begin{alphenumerate}
  \item\label{propCharleftidealcontainingFE1} $\mathscr{L} =
    \mathscr{K}_1;$
  \item\label{propCharleftidealcontainingFE3} $\mathscr{L}$ is a
    non-fixed, finitely-generated, maximal left ideal
    of~$\mathscr{B}(E);$
  \item\label{propCharleftidealcontainingFE2} $\mathscr{L}$ is a
    maximal left ideal of~$\mathscr{B}(E)$ and contains an operator
    which is bounded below.
  \end{alphenumerate}
\end{theorem}
\begin{proof}
  \alphref{propCharleftidealcontainingFE1}$\Rightarrow$%
  \alphref{propCharleftidealcontainingFE3}.  This is immediate
  from Theorem~\ref{Thmnonfixedmaxleftideal}.

  \alphref{propCharleftidealcontainingFE3}$\Rightarrow$%
  \alphref{propCharleftidealcontainingFE2}. Suppose that $\mathscr{L}$
  is a non-fixed, finitely-generated, maximal left ideal
  of~$\mathscr{B}(E)$, so that $\mathscr{L} = \mathscr{L}_{\Gamma}$
  for some non-empty, finite subset~$\Gamma$ of~$\mathscr{B}(E)$. Set
  $n = |\Gamma|\in\N$.  By Corollary~\ref{Sept2011keylemma}, the
  operator~$\Psi_{\Gamma}$ is bounded below.  Moreover, there is an
  operator \mbox{$T\in\mathscr{B}(E^n,E)$} which is bounded below
  because $X_{\text{\normalfont{AH}}}$ embeds in~$\ell_\infty$, and
  $\ell_\infty$ is isomorphic to the direct sum of~$2n-1$ copies of
  itself. Hence the composite operator $T\Psi_{\Gamma}$ is bounded
  below, and it belongs to~$\mathscr{L}$
  by~\eqref{eqFGLeftideal04072012}.

  \alphref{propCharleftidealcontainingFE2}$\Rightarrow$%
  \alphref{propCharleftidealcontainingFE1}. Suppose that $\mathscr{L}$
  is a maximal left ideal of~$\mathscr{B}(E)$ and that $\mathscr{L}$
  contains an~opera\-tor $R= (R_{j,k})_{j,k=1}^2$ which is bounded
  below. Then $R$ does not belong to any fixed maximal left ideal, so
  that $\mathscr{E}(E)\subseteq\mathscr{L}$ by
  Corollary~\ref{June12dich1}.  Lemma~\ref{operatorsXtolinftywc} shows
  that each opera\-tor from~$\ell_\infty$
  to~$X_{\text{\normalfont{AH}}}$ is strictly singular, and thus
  inessential. Hence, by \cite[Proposition~1]{gonzalez}, each operator
  from~$X_{\text{\normalfont{AH}}}$ to~$\ell_\infty$ is also
  inessential, and so we conclude that
  \begin{align}\label{eqinessopinL}
    \biggl\{\begin{pmatrix} T_{1,1} & T_{1,2}\\ T_{2,1} &
      T_{2,2}\end{pmatrix}:\
    &T_{1,1}\in\mathscr{K}(X_{\text{\normalfont{AH}}}),\,
    T_{1,2}\in\mathscr{B}(\ell_\infty,X_{\text{\normalfont{AH}}}),\\
    &T_{2,1}\in\mathscr{B}(X_{\text{\normalfont{AH}}},\ell_\infty),\,
    T_{2,2}\in\mathscr{W}(\ell_\infty) \biggr\} =
    \mathscr{E}(E)\subseteq\mathscr{L}.\notag
  \end{align}
  Since the operator $R$ is bounded below, its restriction
  $R|_{\ell_\infty} = \bigl(\begin{smallmatrix}
    R_{1,2}\\ R_{2,2}\end{smallmatrix}\bigr)$ is also bounded below,
  and is thus an upper semi-Fred\-holm operator. Consequently
  $\bigl(\begin{smallmatrix} 0\\ R_{2,2}\end{smallmatrix}\bigr)$ is an
  upper semi-Fredholm operator by \cite[Proposition~2.c.10]{lt1}
  because~$R_{1,2}$ is strictly singular, and therefore $R_{2,2}$ is
  an upper semi-Fredholm operator. Let $Q\in\mathscr{F}(\ell_\infty)$
  be a projection onto~$\ker R_{2,2}$. Then the restriction of
  $R_{2,2}$ to $\ker Q$ is an isomorphism onto its range, which is a
  closed subspace of~$\ell_\infty$. Since~$\ell_\infty$ is injective,
  the inverse of this isomorphism extends to an operator $S\colon
  \ell_\infty\to \ker Q\subseteq\ell_\infty$, which then satisfies
  $SR_{2,2} = I_{\ell_\infty}-Q$. Hence 
  \[ \begin{pmatrix} 0 &   0\\
    SR_{2,1} & I_{\ell_\infty}-Q\end{pmatrix} = \begin{pmatrix} 0 &
    0\\ 0 & S\end{pmatrix}
  \begin{pmatrix} R_{1,1} & R_{1,2}\\ R_{2,1} &
    R_{2,2}\end{pmatrix}\in\mathscr{L},\] which
  by~\eqref{eqinessopinL} implies that \[
  \begin{pmatrix} 0 & 0\\ 0 &
    I_{\ell_\infty}\end{pmatrix}\in\mathscr{L}. \]
  Applying~\eqref{eqinessopinL} once more, we see that
  $\mathscr{K}_1\subseteq\mathscr{L}$, and so $\mathscr{K}_1 =
  \mathscr{L}$ by the maximality of~$\mathscr{K}_1$.
\end{proof}

\begin{proof}[Proof of Theorem~\ref{Thmnonfixedmaxleftideal2}.]
  The equivalence of
  conditions~\alphref{propCharleftidealcontainingFE1}
  and~\alphref{propCharleftidealcontainingFE3} in
  Theorem~\ref{propCharleftidealcontainingFE} shows
  that~$\mathscr{K}_1$ is the unique non-fixed, finitely-generated,
  maximal left ideal of~$\mathscr{B}(E)$.
  Proposition~{\normalfont{\ref{twosidedidealsinBE}}}%
  \romanref{defnidealI2} implies that $\mathscr{W}_2$ is a maximal
  two-sided ideal. Since
  $\mathscr{F}(E)\subseteq\mathscr{W}_2\nsubseteq\mathscr{K}_1$,
  $\mathscr{W}_2$~is not contained in any finitely-generated, maximal
  left ideal of~$\mathscr{B}(E)$.
\end{proof}

One may wonder whether the conclusion of
Theorem~\ref{Thmnonfixedmaxleftideal} that the ideal~$\mathscr{W}_1$
introduced in
Proposition~\ref{twosidedidealsinBE}\romanref{defnidealI1} is maximal
as a left ideal might be true more generally, that is, not only in the
case where $X$ is Argyros--Haydon's Banach space.  Our next result
implies that this is false for $X = c_0$. Note that all weakly compact
operators on~$c_0$ are compact, so that, in this case, $\mathscr{W}_1$
is equal to
\begin{equation}\label{idealK1c0}
  \mathscr{K}_1 = \biggl\{\begin{pmatrix} T_{1,1} & T_{1,2}\\
    T_{2,1} & 
    T_{2,2} \end{pmatrix}\in\mathscr{B}(c_0\oplus\ell_\infty) :
  T_{1,1}\in\mathscr{K}(c_0)\biggr\}.
\end{equation} 
\begin{proposition}\label{propK1notinfgmaxleftideal}
  The ideal $\mathscr{K}_1$ given by~\eqref{idealK1c0} is not
  contained in any finitely-generated, maximal left ideal
  of~$\mathscr{B}(c_0\oplus\ell_\infty)$.
\end{proposition}

For clarity, we present the main technical step in the proof of
Proposition~\ref{propK1notinfgmaxleftideal} as a separate lemma.

\begin{lemma}\label{lemmaProj}
  Suppose that $T\in\mathscr{B}(c_0)$ is not an upper semi-Fredholm
  operator. Then there exist a projection
  \mbox{$Q_0\in\mathscr{B}(c_0)$} and a normalized basic sequence
  $(x_n)_{n\in\N}$ in~$c_0$ such that $(x_n)_{n\in\N}$ is equivalent
  to the standard unit vector basis for~$c_0$ and
  \begin{equation}\label{lemmaProjEq1}
    Q_0x_{2n-1} = x_{2n-1},\quad Q_0x_{2n} = 0,\quad
    \text{and}\quad \| 
    Tx_n\|\le\frac{1}{n}\quad (n\in\N). \end{equation}
\end{lemma}
\begin{proof} Let $(e_n)_{n\in\N}$ denote the standard unit vector
  basis for~$c_0$.  Since $T$ is not an upper semi-Fredholm operator,
  there are two cases to consider.

  \emph{Case 1:} $\dim\ker T = \infty$. Then $\ker T$ contains a
  closed subspace~$Y$ which is isomorphic to~$c_0$ and complemented
  in~$c_0$ (\emph{e.g.}, see \cite[Proposition~2.a.2]{lt1}).  Let
  $(x_n)_{n\in\N}$ be a normalized Schauder basis for~$Y$ such that
  $(x_n)_{n\in\N}$ is equivalent to~$(e_n)_{n\in\N}$.  Since $Y$ is
  complemented in~$c_0$ and the basis~$(x_n)_{n\in\N}$ is
  unconditional, there is a projection $Q_0\in\mathscr{B}(c_0)$ which
  satisfies the first two identities in~\eqref{lemmaProjEq1}, while
  the third one is trivial because $x_n\in Y\subseteq\ker T$ for each
  $n\in\N$.

  \emph{Case 2:} $\dim\ker T < \infty$ and $T(c_0)$ is not closed.
  For each $n\in\N$, choose $\epsilon_n\in(0,1)$ such that
  $(1+\|T\|)\epsilon_n(1-\epsilon_n)^{-1}\le n^{-1}$.  By induction,
  we shall construct a normalized block basic
  sequence~$(x_n)_{n\in\N}$ of $(e_n)_{n\in\N}$ such that $\|
  Tx_n\|\le n^{-1}$ for each $n\in\N$.

  To start the induction, we observe that $T$ cannot be bounded below
  because its range is not closed, so that we can find a unit vector
  $y_1\in c_0$ such that $\|Ty_1\|\le\epsilon_1$. Approximating~$y_1$
  within~$\epsilon_1$ by a finitely-supported vector and normalizing
  it, we obtain a finitely-supported unit vector~$x_1\in c_0$ such
  that $\| Tx_1\|\le (1+\|T\|)\epsilon_1(1-\epsilon_1)^{-1}\le 1$ by
  the choice of~$\epsilon_1$.

  Now assume inductively that unit vectors $x_1,\ldots,x_n\in c_0$
  with consecutive supports have been chosen for some $n\in\N$ such
  that $\|Tx_j\|\le 1/j$ for each $j\in\{1,\ldots,n\}$. Let $m\in\N$
  be the maximum of the support of~$x_n$, so that
  $x_1,\ldots,x_n\in\operatorname{span}\{e_1,\ldots,e_m\}$, and let
  $P_m$ be the $m^{\text{th}}$ basis projection associated with
  $(e_j)_{j\in\N}$.  If $T|_{\ker P_m}$ were bounded below, then it
  would have closed range, so that \[ T(c_0) = T(\ker P_m) +
  \operatorname{span}\{Te_1,\ldots,Te_m\} \] would also be closed,
  being the sum of a closed subspace and a finite-dimensional
  one. This is false, and hence $T|_{\ker P_m}$ is not bounded
  below. We can therefore choose a unit vector~$y_{n+1}\in \ker P_m$
  such that $\| Ty_{n+1}\|\le\epsilon_{n+1}$. Now, as in the first
  step of the induction, we approximate $y_{n+1}$
  within~$\epsilon_{n+1}$ by a finitely-supported vector in~$\ker P_m$
  and normalize it to obtain a finitely-supported unit vector
  $x_{n+1}\in\ker P_m$ such that \[ \| Tx_{n+1}\|\le
  \frac{(1+\|T\|)\epsilon_{n+1}}{1-\epsilon_{n+1}}\le \frac{1}{n+1} \]
  by the choice of~$\epsilon_{n+1}$.  Therefore the induction
  continues.

  By~\cite[Proposition~2.a.1]{lt1}, the sequence $(x_n)_{n\in\N}$ is
  equivalent to~$(e_n)_{n\in\N}$, and its closed linear span is
  complemented in~$c_0$. Hence, as in Case~1, we obtain a projec\-tion
  \mbox{$Q_0\in\mathscr{B}(c_0)$} such that the first two identities
  in~\eqref{lemmaProjEq1} are satisfied, while the third one holds by
  the con\-struc\-tion of~$(x_n)_{n\in\N}$.
\end{proof}

\begin{proof}[Proof of
  Proposition~{\normalfont{\ref{propK1notinfgmaxleftideal}}}.]
  Assume towards a contradiction that $\mathscr{L}$ is a
  finite\-ly-generated, maximal left ideal of
  \mbox{$\mathscr{B}(c_0\oplus\ell_\infty)$} such that
  $\mathscr{K}_1\subseteq\mathscr{L}$.
  Proposition~\ref{cartesianSGnew} implies that $\mathscr{L}$ is
  generated by a single operator, say
  \[ T = \begin{pmatrix} T_{1,1} & T_{1,2}\\ T_{2,1} &
    T_{2,2} \end{pmatrix}\in\mathscr{B}(c_0\oplus\ell_\infty).  \] 

  We \emph{claim} that~$T_{1,1}$ is not an upper semi-Fredholm
  operator. Assume the contrary; that is, $\ker T_{1,1}$ is
  finite-dimensional, so that we can take a
  projection~$P\in\mathscr{F}(c_0)$ onto $\ker T_{1,1}$, and
  $T_{1,1}(c_0)$ is closed. Then the restriction
  $\widetilde{T}_{1,1}\colon x\mapsto T_{1,1}x,\, \ker P\to
  T_{1,1}(c_0)$, is an isomorphism. Its range is complemented in~$c_0$
  by Sobczyk's theorem~\cite{sob} because it is isomorphic to~$\ker
  P$, which is a closed subspace of finite codimension in~$c_0$, and
  hence isomorphic to~$c_0$.  We can therefore extend the inverse
  of~$\widetilde{T}_{1,1}$ to an operator $S\in\mathscr{B}(c_0)$,
  which then satisfies $ST_{1,1} = I_{c_0} - P$. Since $P$ has finite
  rank, we have
  \[ \begin{pmatrix} P & 0 \\ 0 &
    I_{\ell_\infty} \end{pmatrix}\in\mathscr{K}_1\subseteq
  \mathscr{L}\] and \[
  \begin{pmatrix} T_{1,1} & 0 \\ 0 & 0 \end{pmatrix} = T
  - \begin{pmatrix} 0 & T_{1,2}\\ T_{2,1} &
    T_{2,2} \end{pmatrix}\in\mathscr{L} -
  \mathscr{K}_1\subseteq\mathscr{L}, \] which implies that
  \[ \begin{pmatrix} I_{c_0} & 0 \\ 0 & I_{\ell_\infty} \end{pmatrix}
  = \begin{pmatrix} P & 0 \\ 0 & I_{\ell_\infty} \end{pmatrix}
  + \begin{pmatrix} S & 0 \\ 0 & 0 \end{pmatrix}
  \begin{pmatrix} T_{1,1} & 0 \\ 0 & 0 \end{pmatrix}\in\mathscr{L}. \]
  This, however, contradicts the assumption that the left ideal
  $\mathscr{L}$ is proper, and thus completes the proof of our claim.

  Hence, by Lemma~\ref{lemmaProj}, we obtain a projection
  \mbox{$Q_0\in\mathscr{B}(c_0)$} and a normalized basic sequence
  $(x_n)_{n\in\N}$ in~$c_0$ such that $(x_n)_{n\in\N}$ is equivalent
  to the standard unit vector basis $(e_n)_{n\in\N}$ for~$c_0$ and
  \begin{equation}\label{lemmaProjEq1'}
    Q_0x_{2n-1} = x_{2n-1},\quad Q_0x_{2n} = 0,\quad
    \text{and}\quad \| 
    T_{1,1}x_n\|\leq\frac{1}{n}\quad (n\in\N). \end{equation}
  The sequence  $(x_n)_{n\in\N}$ is weakly null 
  because it is equivalent to the weakly null
  se\-quence $(e_n)_{n\in\N}$, and so the sequence $(Rx_n)_{n\in\N}$ is
  norm-null for each $R\in\mathscr{K}(c_0)$. Now let
  \[ Q =
  \begin{pmatrix} Q_0 & 0\\ 0 &
    0\end{pmatrix}\in\mathscr{B}(c_0\oplus\ell_\infty). \] The
  maximality of the left ideal~$\mathscr{L}$ implies that either
  \[ \text{(i)}\quad Q\in\mathscr{L}\quad\quad
  \text{or}\quad\quad \text{(ii)}\quad \mathscr{L} +
  \mathscr{L}_{\{Q\}} =
  \mathscr{B}(c_0\oplus\ell_\infty). \] We shall complete the proof by
  showing that both of these alternatives are impossible.

  In case (i), there is $S =
  (S_{j,k})_{j,k=1}^2\in\mathscr{B}(c_0\oplus\ell_\infty)$ with $Q =
  ST$.  Defining $P_0\in\mathscr{B}(c_0\oplus\ell_\infty,c_0)$ by
  $P_0(x,f) = x$ for each $x\in c_0$ and $f\in\ell_\infty$, we have
  \[ x_{2n-1} = P_0Q(x_{2n-1},0) = P_0ST(x_{2n-1},0) =
  S_{1,1}T_{1,1}x_{2n-1} + S_{1,2}T_{2,1}x_{2n-1} \] for each
  $n\in\N$.  This, however, is absurd since the left-hand side is a
  unit vector, whereas the right-hand side norm-converges to~$0$ as
  $n\to\infty$: this holds because $\|T_{1,1}x_{2n-1}\|\to 0$
  by~\eqref{lemmaProjEq1'} and $S_{1,2}T_{2,1}\in\mathscr{W}(c_0) =
  \mathscr{K}(c_0)$.

  In case (ii), there are operators
  $U,V\in\mathscr{B}(c_0\oplus\ell_\infty)$  such that 
  $$I_{c_0\oplus\ell_\infty} = UT + VQ.$$ Define~$P_0$ as
  above, and write $U = (U_{j,k})_{j,k=1}^2$. Then, since
  $Q_0x_{2n}=0$, we have
  \[ x_{2n} = P_0(UT +VQ)(x_{2n},0) = 
  U_{1,1}T_{1,1}x_{2n} + U_{1,2}T_{2,1}x_{2n}\quad (n\in\N), \] which
  leads to a contradiction as in case~(i) because the left-hand side
  is a unit vector, whereas the right-hand side norm-converges to~$0$
  as $n\to\infty$.
\end{proof}

\section*{Acknowledgements} 
\noindent
We are grateful to David Blecher (Houston) and Mikael R\o{}rdam
(Copenhagen) for having informed us that the conjecture of the first
author and \.{Z}elazko stated on p.~\pageref{DZconj} has a positive
answer for $C^*$-algebras, and for having shown us how to deduce it
from standard facts. We are also grateful to the referee for
his/her careful reading of our paper which has led to simplified
proofs of several results in Section~\ref{sectionLcontainsFE}.

Key parts of this work were carried out during a visit of Piotr
Koszmider to Lancaster in February 2012. The visit was supported by a
London Mathematical Society Scheme 2 grant (ref.~21101). The authors
gratefully acknowledge this support.


\begin{thebibliography}{99}
\bibitem{ak}  F.~Albiac and N.~J.~Kalton, \emph{Topics in Banach space
  theory},  Grad.\ Texts in Math.~233, Springer-Verlag, New York, 2006.
\bibitem{amir} D.~Amir, \emph{Continuous function spaces with the bounded
  extension property}, {Bull.\ Res.\ Concil Israel
    Sect.~F}~{101} (1962), 133--138.
\bibitem{argyrosetal} S.~A.~Argyros, J.~F.~Castillo, A.~S.~Granero,
  M.~Jim\'{e}nez, and J.~P.~Moreno, \emph{Complementation and embeddings of
  $c_0(I)$ in Banach spaces,} {Proc.~London Math.\
    Soc.}~{85} (2002), 742--768.
\bibitem{ah} S.~A.~Argyros and R.~G.~Haydon, \emph{A hereditarily
    indecomposable $\mathscr{L}_\infty$-space that solves the
    scalar-plus-compact problem}, {Acta Math.}~{206} (2011), 1--54.
\bibitem{baker} J.~W.~Baker, \emph{Some uncomplemented subspaces of $C(X)$
  of the type~$C(Y)$}, {Studia Math.}~{36} (1970), 85--103.
\bibitem{BlecherKania} D.~Blecher and T.~Kania, \emph{Finite generation in
  $C^*$-algebras and Hilbert $C^*$-modules}, in preparation.
\bibitem{boudi} N.~Boudi, \emph{Banach algebras in which every left ideal is
  countably generated}, {Irish Math.\ Soc.\ Bull.}~{48}
  (2002), 17--24.
\bibitem{cpy} S.~R.~Caradus, W.~E.~Pfaffenberger and B.~Yood,
  \emph{Calkin algebras and algebras of operators on Banach spaces},
  Lect.\ Notes Pure Appl.\ Math.~9, Marcel Dekker, New York,
  1974.
\bibitem {cn} W.~W.~Comfort and S.~Negrepontis, \emph{The theory of
    ultrafilters}, Grundlehren Math.\ Wiss.~211, Springer-Verlag, New
  York--Heidelberg, 1974.
\bibitem{cohen} H.~B.~Cohen, \emph{Injective envelopes of Banach spaces},
  {Bull.\ Amer.\ Math.\ Soc.}~{70} (1964), 723--726.
\bibitem{dales} H.~G.~Dales, \emph{Banach algebras and automatic
    continuity}, London Math.\ Soc.\ Monogr.\ Ser.~24, Clarendon Press,
  Oxford, 2000.   
\bibitem{daleszelazko} H.~G.~Dales and W.~\.{Z}elazko,
  \emph{Generators of maximal left ideals in Banach algebras}, {Studia
    Math.}~{212} (2012), 173--193.
\bibitem{dfjp} W.~J.~Davis, T.~Figiel, W.~B.~Johnson, and
  A.~Pe{\l}czy{\a'n}ski, \emph{Factoring weakly compact operators},
  {J.~Funct.\ Anal.}~{17} (1974), 311--327.
\bibitem{em} I.~S.~Edelstein and B.~S.~Mityagin, \emph{Homotopy type of
  linear groups of two classes of Banach spaces}, {Funct.\
    Anal.~Appl.}~{4} (1970), 221--231.
\bibitem{ft} V.~Ferreira and G.~Tomassini, \emph{Finiteness properties of
  topological algebras}, {Ann.\ Sc.\ Norm.\ Sup.\
    Pisa Cl.\ Sci.}~{5} (1978), 471--488.
\bibitem{gillman} L.~Gillman, \emph{Countably generated ideals in rings of
  continuous functions}, {Proc.\ Amer.\ Math.\ Soc.}~{11}
  (1960), 660--666.
\bibitem{gonzalez} M.~Gonz\'{a}lez, \emph{On essentially incomparable
    Banach spaces}, {Math.~Z.}~{215} (1994), 621--629.
\bibitem{goodner} D.~B.~Goodner, \emph{Projections in normed linear spaces},
  {Trans.\ Amer.\ Math.\ Soc.}~{69} (1950), 89--108.
\bibitem{gm} W.~T.~Gowers and B.~Maurey, \emph{The unconditional basic
    sequence problem}, {J.~Amer.\ Math.\ Soc.}~{6} (1993), 851--874.
\bibitem{gr} B.~Gramsch, \emph{Eine Idealstruktur Banachscher
    Operatoralgebren}, {J.~Reine Angew.\ Math.}~{225} (1967), 97--115.
\bibitem{granero} A.~S.~Granero, \emph{On complemented subspaces
  of~$c_0(I)$}, {Atti Semin.\ Mat.\ Fis.\ Univ.\ M\'{o}dena Reggio Emilia}
  {XLVI} (1998), 35--36.
\bibitem{groth} A.~Grothendieck, \emph{Une caract\'{e}risation
    vectorielle m\'{e}trique des espaces~$L_1$}, {Canad.\
    J.~Math.}~{7} (1955), 552--561.
\bibitem{GronbaekMorita} N.~Gr\o{}nb\ae{}k, \emph{Morita equivalence
    for Banach algebras}, {J.~Pure Appl.\ Algebra}~{99} (1995),
  183--219.
\bibitem{jacobson} N.~Jacobson, \emph{Basic Algebra II},
  W.\ H.\ Freeman, San Francisco, California, 1980.
\bibitem {jech} T.~Jech, \emph{Set Theory}, Third Millennium Ed.,
  revised and expanded, Springer-Verlag, Berlin, 2003.
\bibitem{kr1} R.~V.~Kadison and J.~R.~Ringrose, \emph{Fundamentals of
    the theory of operator alge\-bras~I: Elementary theory},  Pure 
  Appl.\ Math.~100, Academic Press, New York, 1983.
\bibitem{kr2} R.~V.~Kadison and J.~R.~Ringrose, \emph{Fundamentals of
    the theory of operator alge\-bras~II: Advanced theory}, Pure Appl.\
  Math.~100, Academic Press, Orlando, Florida, 1986.
\bibitem{kaltonpeck} N.~J.~Kalton and N.~T.~Peck, \emph{Twisted sums
    of sequence spaces and the three space problem}, {Trans.\ Amer.\
    Math.\ Soc.}~{255} (1979), 1--30.
\bibitem{kanialaustsen} T.~Kania and N.~J.~Laustsen, \emph{Ideal
    structure of the algebra of bounded operators acting on a Banach
    space}, in preparation.
\bibitem{kl} D.~Kleinecke, \emph{Almost-finite, compact, and
    inessential operators}, {Proc.\ Amer.\ Math.\ Soc.}~{14} (1963),
  863--868.
\bibitem{koszmider} P.~Koszmider, \emph{Banach spaces of continuous
    functions with few operators}, {Math.\ Ann.}~{330} (2004),
  151--183.
\bibitem{lau1} N.~J.~Laustsen, \emph{Maximal ideals in the algebra of
    operators on certain Banach spaces}, {Proc.\ Edinb.\ Math.\
    Soc.}~{45} (2002), 523--546.
\bibitem{lau2} N.~J.~Laustsen, \emph{Commutators of operators on Banach
  spaces}, {J.~Operator Theory}~{48} (2002), 503--514.
\bibitem{linden} J.~Lindenstrauss, \emph{On the extension property for
    compact operators}, {Bull.\ Amer.\ Math.\ Soc.}~{68} (1962),
  484--487.
\bibitem{lt1} J.~Lindenstrauss and L.~Tzafriri, \emph{Classical Banach
  spaces~I}, Ergeb.\ Math.\ Grenz\-geb.~{92}, Springer-Verlag,
  Berlin--New York, 1977.
\bibitem{lt2} J.~Lindenstrauss and L.~Tzafriri, \emph{Classical Banach
  spaces~II}, Ergeb.\ Math.\ Grenz\-geb.~{97}, Springer-Verlag,
  Berlin--New York, 1979.
\bibitem{luft} E.~Luft, \emph{The two-sided closed ideals of the algebra of
  bounded linear operators of a Hilbert space}, {Czechoslovak 
    Math.~J.}~{18} (1968), 595--605. 
\bibitem{meg} R.~E.~Megginson, \emph{An introduction to Banach
    space theory}, Grad.\ Texts in Math.~{183}, Springer-Verlag,
  New York, 1998.
\bibitem{murphy} G.~J.~Murphy, \emph{$C^*$-algebras and operator
    theory}, Academic Press, Boston, MA, 1990. 
\bibitem{nachbin} L.~Nachbin, \emph{On the Han--Banach theorem},
  {An.\ Acad.\ Brasil.\ Ci\^{e}nc.}~{21} (1949), 151--154.
\bibitem{pel} A.~Pe{\l}czy{\a'n}ski, \emph{On strictly singular and strictly
  cosingular operators. I. Strictly singular and strictly cosingular
  operators in $C(S)$-spaces}, {Bull.\ Acad.\ Polon.\ Sci.\
    S\'{e}r.\ Sci.\ Math.\ Astronom.\ Phys.}~{13} (1965),
  31--36.
\bibitem{pi} A.~Pietsch, \emph{Operator ideals}, North Holland,
  Amsterdam--New York--Oxford, 1980.
\bibitem{plebanek} G.~Plebanek, \emph{A construction of a Banach space
    $C(K)$ with few operators}, {Topology Appl.}~{143} (2004),
  217--239.
\bibitem{Pospisil} B.~Posp\'{\i}\v{s}il, \emph{On bicompact spaces},
  {Publ.\ Fac.\ Sci.\ Univ.\ Masaryk} {270} (1939), 3--16.
\bibitem{prosser} R.~T.~Prosser, \emph{On the ideal structure of
    operator algebras}, {Mem.\ Amer.\ Math.\ Soc.}~{45} (1963).
\bibitem{Rosenberg} A.~Rosenberg, \emph{The number of irreducible
  representations of simple rings with no minimal ideals},
  {Amer.\ J.~Math.}~{75} (1953), 523--530.
\bibitem{Sch} Th.\ Schlumprecht, \emph{An arbitrarily distortable
    Banach space}, {Israel J.~Math.}~{76} (1991), 81--95.
\bibitem{st} A.~M.~Sinclair and A.~W.~Tullo, \emph{Noetherian Banach
    algebras are finite dimensional}, {Math.\ Ann.}~{211} (1974),
  151--153.
\bibitem{sob} A.~Sobczyk, \emph{Projection of the space $(m)$ on its
    subspace~$(c_0)$}, {Bull.\ Amer.\ Math.\ Soc.}~{47} (1941),
  938--947.
\bibitem{stout} E.~L.~Stout, \emph{The theory of uniform algebras},
  Bogden and Quigley, Tarrytown-on-Hudson, New York, 1971.
\bibitem{Sz} A.~Szankowski, \emph{$\mathscr{B}(H)$ does not have the
  approximation property}, {Acta Math.}~{147} (1981),
  89--108.
\bibitem{yo} B.~Yood, \emph{Difference algebras of linear
transformations on a Banach space}, {Pacific J.~Math.}~{4}
(1954), 615--636.
\end{thebibliography}
\end{document}